%% file: Characteristic_points_on_surfaces2.tex
\documentclass[11pt,english]{article} % ou ‘english’ ou...
\usepackage[utf8]{inputenc}
\usepackage[T1]{fontenc}
\usepackage{lmodern}
\usepackage[a4paper]{geometry}
\usepackage{babel}
\usepackage{amsmath}
\usepackage{cleveref}

\usepackage{nag}
\usepackage{xcolor}

\usepackage{mathtools}
\usepackage{amsthm}
\usepackage{amssymb}
\usepackage{graphicx}
\usepackage{caption}
\usepackage{enumitem}
\usepackage{url}
                \newcommand{\R}{\mathbb R}        
\newcommand{\Z}{\mathbb Z}              \newcommand{\F}{\Phi}      
\newcommand{\f}{\varphi}                \newcommand{\what}{\widehat}
\newcommand{\D}{\partial}               
\newcommand{\s}{\sigma}                 \newcommand{\g}{\gamma} 
             
                 \def\l{\lambda}

\newcommand{\comp}{{\small{\circ}}}     \def\a{\alpha}
\newcommand{\fin}{\hfill{$\square$}}    %\newcommand{\L}{\mathbb L}
            \newcommand{\RP}{{\mathbb R\mathrm P}}
\newcommand{\ind}{\mathrm{ind}}         \newcommand{\sign}{\mathrm{sign}}

\def\pd#1#2{\frac{\D#1}{\D#2}}          \newcommand{\Sym}{\mathrm{Sym}}
\def\t{\tau}

\begin{document}

\captionsetup[figure]{labelfont={bf},labelformat={default},labelsep=period,name={\small Fig.}}

%%%%%%%%%%%%%%%%%%%%%%%%%%%%%%%%%%%%%%%%%%%%%%%%%%%%%%%%%%%%%%%%%%%%%%%%%%%%%%%%%%%%%%%%%%%%%%%%%%%%%%%%
%%%%%%%%%%%%%%%%ù     STYLE  THEOREM    %%%%%%%%%%%%%%%%%%%%%%%%%%%%%%%%%%%%%%%%%%%%%%%%%%%%%%%%%%%%%%%%
%%%%%%%%%%%%%%%%%%%%%%%%%%%%%%%%%%%%%%%%%%%%%%%%%%%%%%%%%%%%%%%%%%%%%%%%%%%%%%%%%%%%%%%%%%%%%%%%%%%%%%%%

%\numberwithin{equation}{section}                % numerazione delle equazioni
\theoremstyle{plain}% default
\newtheorem{theorem}{\bf Theorem}
\newtheorem*{theorem*}{\bf Theorem}
\newtheorem{teorema}{\bf Theorem}
\def\theteorema{\Roman{teorema}}
\newtheorem{lemma}{\bf Lemma}
\newtheorem*{lemma*}{\bf Lemma}
\newtheorem*{mlemma}{\bf Main Lemma}

\newtheorem{proposition}{\bf Proposition}
\newtheorem*{proposition*}{\bf Proposition}
\newtheorem{corollary}{\bf Corollary}
\newtheorem*{corollary*}{\bf Corollary}
\newtheorem*{conjecture}{\bf Conjecture}%[section]
\newtheorem*{fact*}{\bf Fact}
\newtheorem{property}{\bf Property}

%%%%%%%%%%%%%%%%%%%%%%%%%%%%%%%%%%%%%%%%%%%%%%%%%%%%%%%%%%%%%%%%%%%%%%%%%%%%%%%%%%%%%%%%%%%%%%%%%%%%%%%%
%%%%%%%%%%%%%%%%ù     STYLE DEFINITION    %%%%%%%%%%%%%%%%%%%%%%%%%%%%%%%%%%%%%%%%%%%%%%%%%%%%%%%%%%%%%%
%%%%%%%%%%%%%%%%%%%%%%%%%%%%%%%%%%%%%%%%%%%%%%%%%%%%%%%%%%%%%%%%%%%%%%%%%%%%%%%%%%%%%%%%%%%%%%%%%%%%%%%%

\theoremstyle{definition}
\newtheorem{definition}{\bf Definition}%[section]
\newtheorem*{definition*}{\bf Definition}

\newtheorem{example}{\bf Example}
\newtheorem*{example*}{\bf Example}
\newtheorem{Example}{\bf Example}[section]
\theoremstyle{remark}
\newtheorem*{remark*}{\bf Remark}
\newtheorem{remark}{\bf Remark}
\newtheorem*{problem}{\bf Problem}

%%%%%%%%%%%%%%%%%%%%%%%%%%%%%%%%%%%%%%%%%%%%%%%%%%%%%%%%%%%%%%%%%%%%%%%%%%%%%%%%%%%%%%%%%%%%%%%%%%%%%%%%
%%%%%%%%%%%%%%%%ù     TITLE    %%%%%%%%%%%%%%%%%%%%%%%%%%%%%%%%%%%%%%%%%%%%%%%%%%%%%%%%%%%%%%%%%%%%%%%%%
%%%%%%%%%%%%%%%%%%%%%%%%%%%%%%%%%%%%%%%%%%%%%%%%%%%%%%%%%%%%%%%%%%%%%%%%%%%%%%%%%%%%%%%%%%%%%%%%%%%%%%%%

\title{\vspace{-2.5cm} {\small To appear in \textsc{Moscow Mathematical Journal}
\textbf{20}:3, July-August 2020} \\ 
{\footnotesize A talk in \url{https://www.youtube.com/watch?v=HIG-Eu8dqKI}} \\
\ \\ Characteristic Points, Fundamental Cubic Form \\
and Euler Characteristic of Projective Surfaces}

\author{By \ Maxim\,Kazarian\footnote{National Research University Higher School of Economics, Moscow, Russia; and  Skolkovo Institute of Science and Technology, Moscow, Russia. \emph{kazarian@mccme.ru}},
  \ and \ Ricardo\,Uribe-Vargas\footnote{Laboratory Solomon Lefschetz UMI2001 CNRS, 
  Universidad Nacional Autonoma de M\'exico, M\'exico City; and Institut de Math\'ematiques de Bourgogne, UMR 5584, CNRS, Universit\'e 
  Bourgogne Franche-Comt\'e, F-21000 Dijon, France. \emph{r.uribe-vargas@u-bourgogne.fr}}}

\date\empty                     % Data
\maketitle

%\addtolength{\topmargin}{-67pt}
%\addtolength{\textheight}{100pt}
%\oddsidemargin=0pt
%\textwidth=6in

%\newtheorem{theorem}{Theorem}[section]
%\newtheorem{lemma}[theorem]{Lemma}
%\newtheorem{conjecture}[theorem]{Conjecture}
%\newtheorem{proposition}[theorem]{Proposition}
%\newtheorem{corollary}[theorem]{Corollary}
%{\theorembodyfont{\rmfamily}
%\newtheorem{remark}[theorem]{Remark}
%\newtheorem{example}[theorem]{Example}
%\newtheorem{definition}[theorem]{Definition}
%\newtheorem{problem}[theorem]{Problem}
%}

\let\h\theta
\let\t\tau

\begin{abstract}
\noindent
We define local indices for %three kinds of characteristic points 
projective umbilics and godrons (also called cusps of Gauss) on generic smooth surfaces
in projective 3-space. By means of these indices, we provide formulas that relate the algebraic
numbers of those characteristic points on a surface (and on domains of the surface) with the
Euler characteristic of that surface (resp. of those domains).
These relations determine the possible coexistences of %the different characteristic points
projective umbilics and godrons on the surface. 
Our study is based on a ``fundamental cubic form'' for which we provide a simple expression. 
%Global counting of these invariants leads to the formulas for the Euler
%characteristic of the surface and for domains in it. These formulas
%provide restrictions on the coexistence of special points.
\end{abstract}
\smallskip

\smallskip


%%%%%%%%%%%%%%%%%%%%%%%%%%%%%%%%%%%%%%%%%%%%%%%%%%%%%%%%%%%%%%%%%%%%%%%%%%%%%%%%%%%%%%%%%%%%%%%%%%%%%%%%
%%%%%%%%%%%%%%%%%     BODY     %%%%%%%%%%%%%%%%%%%%%%%%%%%%%%%%%%%%%%%%%%%%%%%%%%%%%%%%%%%%%%%%%%%%%%%%%
%%%%%%%%%%%%%%%%%%%%%%%%%%%%%%%%%%%%%%%%%%%%%%%%%%%%%%%%%%%%%%%%%%%%%%%%%%%%%%%%%%%%%%%%%%%%%%%%%%%%%%%%

%%%%%%%%%%%%%%%%%%%%%%%%%%%%%%%%%%%%%%%%%%%%%%%%%%%%%%%%%%%%%%%%%%%%%%%%%%%%%%%%%%%%%%%%%%%%%%%%%%%%%%%
%%%%%%%%%%%%%%%%%%%%%%%%%%%%%%%%%%%%%%%%%%%%%%%%%%%%%%%%%%%%%%%%%%%%%%%%%%%%%%%%%%%%%%%%%%%%%%%%%%%%%%%
\section{Introduction}
%%%%%%%%%%%%%%%%%%%%%%%%%%%%%%%%%%%%%%%%%%%%%%%%%%%%%%%%%%%%%%%%%%%%%%%%%%%%%%%%%%%%%%%%%%%%%%%%%%%%%%%

\paragraph{1.1 Counting Characteristic Points.}
The classical M\"obius theorem asserts that a non-contractible curve $C$ embedded in the projective plane
has at least three \textit{inflections} (points where the curve has unusual tangency with its
tangent line). Its \textit{dual curve}, denoted $C^\vee\subset(\RP^2)^\vee$,
consists of the tangent lines to $C$.  
{\em The higher contact of $C$ with its tangent line at an inflection is expressed as a cusp of $C^\vee$} 
(Fig.\,\ref{moebius}).

\begin{figure}[h]
\centering
\begin{minipage}[t]{5cm}
\centering
\includegraphics [scale=0.035]{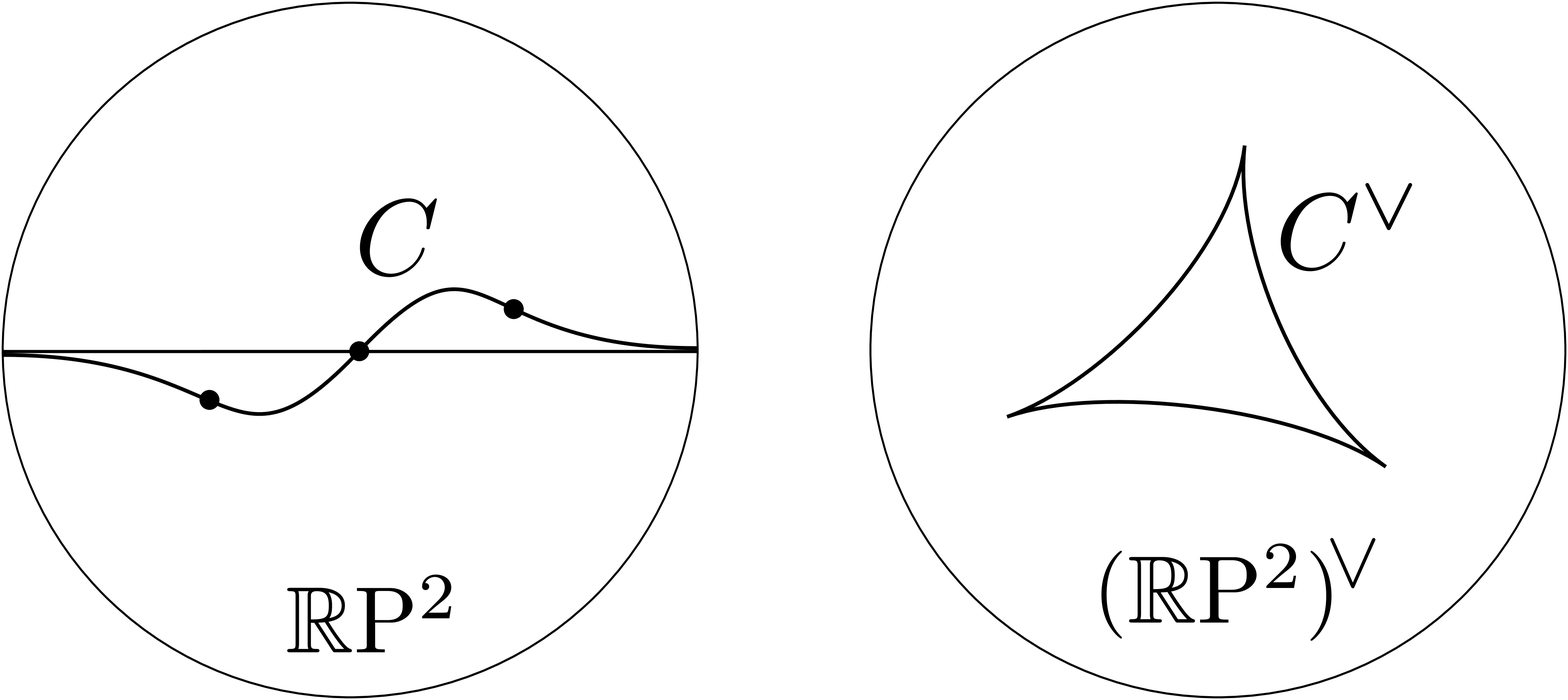}
\caption{\small M\"oebius Theorem.}
\label{moebius}
\end{minipage}
\hspace{1cm}
\begin{minipage}[t]{6cm}
\centering
\includegraphics [scale=0.3]{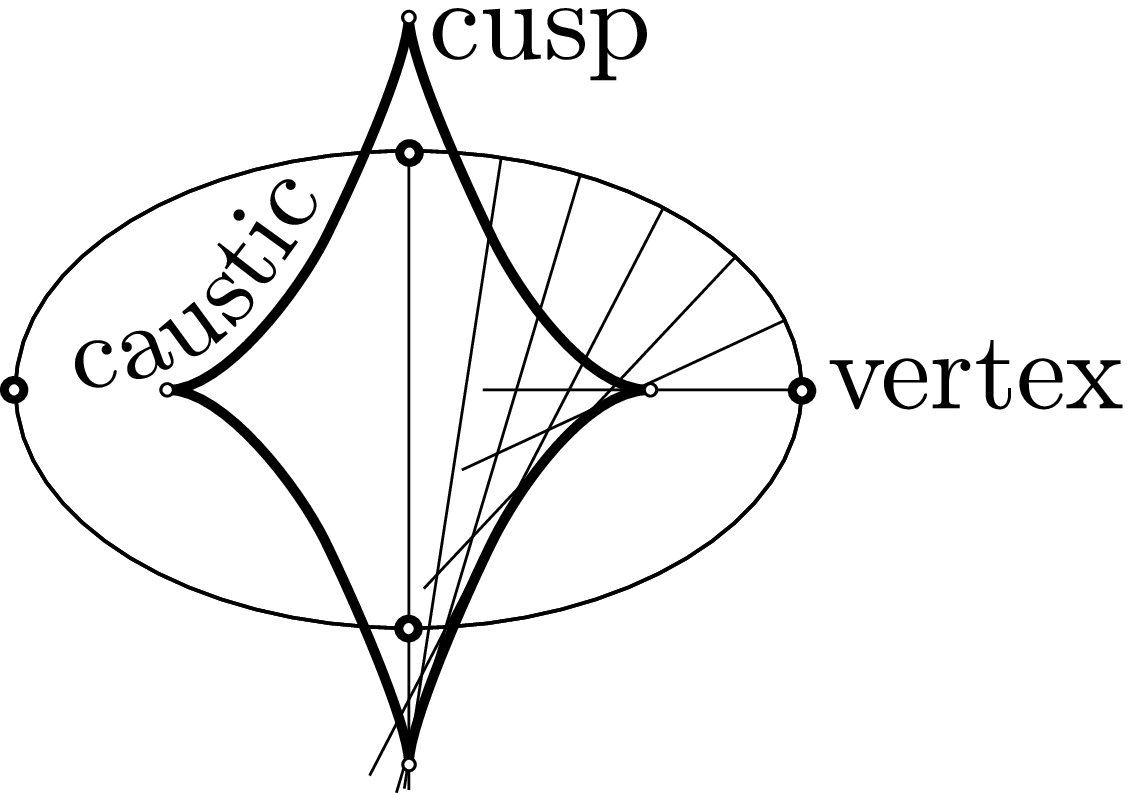}
\caption{\small Four-vertex Theorem.}
\label{4-vertex}
\end{minipage}
\end{figure}

\noindent
The classical $4$-Vertex Theorem (S. Mukhopadhyaya, 1909) asserts that an embedded closed curve in Euclidean plane has
at least $4$ \textit{vertices} (points where the curve has unusual tangency with its
osculating circle).  
The \textit{caustic} (or \textit{evolute}) of a curve $C$ of Euclidean plane is the envelope of its 
normal lines. %(and consists of the centers of curvature of $C$). 
{\em The higher contact of $C$ with its osculating circle at a vertex 
is expressed as a cusp of the caustic} (Fig.\,\ref{4-vertex}).  

Mukhopadhyaya also proved the statement that a convex closed plane curve has
at least $6$ \textit{sextactic points} (where the curve has unusual tangency with its
osculating conic).

The recent activity on variations and generalisations of these theorems
(cf. \cite{Kazarian-Flat, Dima, Uribeconformal, Chekanov-Pushkar, Ovsienko-Tabachnikov}) 
%, also for curves and surfaces in higher dimensional spaces,
was in part stimulated by the works of V.\,I.\,Arnold who presented these results as
special cases of general theorems of contact and symplectic topology, in which 
singularity theory plays an important r\^ole (cf. \cite{Arnold-RPCS, Arnold-LSTSC, Arnold-TPTAC}). 

In this paper, we provide a ``global counting'' of projective umbilics and godrons (both defined below) 
%of  consider characteristic points 
of generic smooth surfaces of $\RP^3$ (and $\R^3$), and we state some coexistence relations\,: 

%For a generic smooth surface of $\RP^3$ different kinds of isolated points (and some salient curves)
%stable under small perturbations.
%We shall state some relations between the numbers of such points and the topological type of the
%surface, implying that those numbers are far to be arbitrary. Let us present our results.

A generic smooth surface of $\RP^3$ (or $\R^3$) consists of three parts, may be empty\,:
an open \textit{elliptic domain} at which the second fundamental form $Q$ is definite
(the Gaussian curvature $K$ being positive); an open {\em hyperbolic domain}
where the form $Q$ is indefinite ($K$ being negative); 
and a \textit{parabolic curve} $P$ where $Q$ is degenerate ($K=0$).
The two lines on which $Q$ vanish at a given hyperbolic point are called  \textit{asymptotic lines} of
the surface at that point. At the parabolic points there is a unique (but double) asymptotic line. 

Each of the above three parts contains characteristic isolated points\,: 
a \textit{godron} (or \textit{cusp of Gauss}) is a parabolic point at which the unique (but double) asymptotic line is
tangent to the parabolic curve; a \textit{hyperbolic} or an
\textit{elliptic projective umbilic}\footnote{Projective umbilics are also called \textit{quadratic points}
  (see \cite{Ovsienko-Tabachnikov}).} 
(or \textit{node}) is a point where the surface is approximated by a quadric up to order $3$. 
We also call \textit{hyperbonodes} the hyperbolic projective umbilics and
\textit{ellipnodes} the elliptic ones.
\medskip

{\footnotesize
\noindent
Thus ellipnodes and hyperbonodes of surfaces are the analogues of sextactic points of curves.}
\medskip

Each godron of a generic surface has an intrinsic index with value $-1$ or $+1$
(cf. \cite{Dima, Uribegodron}). Below, we characterise the hyperbonodes and ellipnodes as the
singular points of an intrinsic field of (triples of) lines, which ascribes also to them an index.

A godron, a hyperbonode or an ellipnode is said to be \textit{positive} (or \textit{negative})
if its index is positive (resp. negative).  Let us state our main result.
%(Several characterisations of positive and negative godrons are given in \cite{Uribegodron}.) 

Given a generic smooth compact surface $S$ of $\RP^3$, let $H$ be a connected component
of the hyperbolic domain and $E$ a connected component of the elliptic domain.

Write $\#e(E)$, $\#g(E)$, $\#h(H)$ and $\#g(H)$ for the respective algebraic numbers of ellipnodes in $E$,
godrons on $\D E$, hyperbonodes in $H$ and godrons on $\D H$. 
\begin{theorem}\label{main-theorem}
For a generic surface $S$ of $\RP^3$ the following three equalities hold 
\begin{enumerate}[label={\rm (\alph*)}, parsep=0.124cm, itemsep=0.0cm,topsep=0.2cm]\itemsep=0.0cm  
\item  $\#h(H)\,=\,\chi(H)$;
\item  $\#g(\D H)\,=\,2\,\chi(H)$; 
\item  $\#e(E)\,-\,\#g(\D E)\,=\,3\,\chi(E)$.
\end{enumerate}
\end{theorem}

Summing up the three equalities of Theorem~\ref{main-theorem} over all connected
components of both hyperbolic and elliptic domains, we obtain the following
relation derived first in\,\cite{Dima}.

\begin{corollary}\label{3Euler}
  For a generic surface $S$ of $\RP^3$ the sum of the algebraic numbers of ellipnodes and hyperbonodes on $S$
  is thrice the Euler characteristic of $S$\,: 
\[\#e(S)+\#h(S)=3\chi(S)\,.\]
\end{corollary}

{\footnotesize
\begin{example*}
A surface diffeomorphic to a sphere has at least $6$ projective umbilics. 
For example, an ovaloid of a cubic surface has exactly $6$ positive ellipnodes \cite{Dima}. 
Consider a local continuous deformation that produces a small hyperbolic island $H$ 
(Fig.\,\ref{fig-six-umbilics}-left). 
By Theorem\,\ref{main-theorem} a and b, $H$ has one positive hyperbonode 
and its boundary $\D H$ has two positive godrons (in the simplest case). 
Then, by Theorem\,\ref{main-theorem} c, the elliptic domain has five positive 
ellipnodes (in the simplest case). 
The local transition is described in Fig.\,\ref{fig-six-umbilics}-right.

\begin{figure}[h] %< préférence de placement h , t , b or p >
\centering
\includegraphics[scale=0.3]{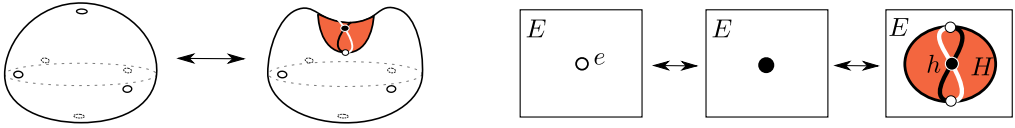}
\caption{\small Left: an ovaloid and its deformation. \ Right: a generic local 
transition (\cite{Uribeevolution}).}
\label{fig-six-umbilics}
\end{figure}
\end{example*}
}

%%%%%%%%%%%%%%%%%%%%%%%%%%%%%%%%%%%%%%%%%%%%%%%%%%%%%%%%%%%%%%%%%%%%%%%%%%%%%%%%%%%%%%%%%%%%%%%%%%%%%%%%%%%%%%%%%%
\paragraph{1.2 Fundamental Cubic Form.} To prove Theorem\,\ref{main-theorem} we characterise
the projective umbilics as the singular points of the field of zeroes of a ``fundamental cubic form''.

Similarly as the second fundamental form describes the quadratic deviation of a surface from 
its tangent plane, the fundamental cubic form describes the cubic deviation of the surface 
from its quadratic part. Let us introduce our notations. 
\smallskip

\noindent
\textit{\textbf{\small Monge form}}. To give an expression of the fundamental cubic form, we identify 
the affine chart of the projective space,
$\{[x : y : z : 1]\}\subset\RP^3$ with $\R^3$, and present the germs of surfaces %in $\R^3$ 
at the origin in Monge form \,$z=f(x,y)$\, with
\,$f(0,0)=0$\, and \,$df(0,0)=0$.
\smallskip

We shall express the partial derivatives of $f$ with numerical subscripts\,:
\[f_{ij}(x,y):=\frac{\D^{i+j}f}{\D x^i\D y^j}(x,y) \quad \text{and} \quad f_{ij}:=f_{ij}(0,0)\,.\]

Take the Taylor expansion $f=Q+C+\ldots$, where $Q$ is the quadratic part and $C$ the cubic part.  
The discriminant of the quadratic form $Q$ (multiplied by $4$) is given by the Hessian 
\[H=f_{20}(x,y)f_{02}(x,y)-f_{11}^2(x,y)\,\]
(which is positive at the elliptic points and negative at the hyperbolic points, 
so that the parabolic curve is given by the equation $H=0$).
\medskip

\noindent
\textbf{\small Formula for the Fundamental Cubic Form}.
The fundamental cubic form of a surface is the homogeneous degree~$3$ form on the tangent space given by
\[W= 4H C-Q\,dH\]
where $dH$ is the linear part of the Taylor expansion of the function~$H$.
\medskip

In the hyperbolic domain, the zeroes of the form $W$ define a field $\t$ of lines whose singular points 
are the hyperbonodes; while in the elliptic domain, the zeroes of $W$ define a field $\t$ of triples of lines 
whose singular points are the ellipnodes. 
%%%%%%%%%%%%%%%%%%%%%%%%%%%%%%%%%%%%%%%%%%%%%%%%%%%%%%%%%%%%%%%%%%%%%%%%%%%%%%%%%%%%%%%%%%%%%%%%%%%%%%%%%%%%%%%%%%
\paragraph{1.3 Expressions for the Indices.}
Take the asymptotic lines as coordinate axes. At a hyperbonode $h$, the index of the field $\t$ of lines is 
given by (Proposition\,\ref{local-index-h}): 
\begin{equation}\label{expression-index-h}
\ind_h(\t)=1\cdot\sign\left(4f_{11}^2f_{40}f_{04}-(2f_{11}f_{31}-3f_{21}^2)(2f_{11}f_{13}-3f_{12}^2)\right)\,.
\end{equation}

At an ellipnode $e$ for which $Q=\frac{\a}{2}(x^2+y^2)$, and the cubic terms are absent, 
the index of the field $\t$ of triples of lines is given by the expression (Proposition\,\ref{local-index-e}):
\[\ind_e(\t)=\frac{1}{3}\cdot\sign\left((f_{31}-3f_{13})(f_{13}-3f_{31})-(f_{40}-3f_{22})(f_{04}-3f_{22})\right)\,.\] 

\begin{remark*}[\textbf{on a cross-ratio invariant}]
At a hyperbonode $h$ the tangent lines to the flecnodal curves and the asymptotic lines define 
a cross-ratio invariant, that we note $\rho(h)$. The asymptotic lines at $h$ 
have $4$-point contact with the surface. We say that $h$ has \textit{parity} $\s(h)=+1$ if both asymptotic 
lines locally lie on the same side of the surface and $\s(h)=-1$ otherwise. 

Taking the asymptotic lines as coordinate axes, $\rho(h)$ and $\s(h)$ are given by (\cite{Uribeinvariant}):
\begin{align*}
\rho(h)&=1-\frac{(3f_{21}^2-2f_{11}f_{31})(3f_{12}^2-2f_{11}f_{13})}{4f_{11}^2f_{40}f_{04}}\,,\\
\s(h)&= \sign(f_{40}f_{04})\,. 
\end{align*}

From these expressions and equality \eqref{expression-index-h} we get the formula 
\[\ind_h(\t)=\sign(\rho(h)\s(h))\,.\] 

A similar formula holds for the ellipnodes. 
\end{remark*}

{\footnotesize
\noindent
\textbf{Organisation of the paper}. In \S\,\ref{section-properties}, we recall some basic 
properties of generic smooth surfaces related to the characteristic points.
In \S\,\ref{section-FCF}, we give a formal definition of the fundamental cubic form 
and the proofs of its properties. In \S\,\ref{Section-Poincare-Hopf}, we extend the statement of  
Poincaré-Hopf theorem to the case of “multivalued fields of lines” on a surface with boundary. 
In \S\,\ref{Section-Proof-Theorem}, we prove Theorem\,\ref{main-theorem}. 
In \S\,\ref{Section-Local-Indices}, we compute the index of the field of (triples of) lines $\t$ at 
hyperbonodes, ellipnodes and godrons (as boundary singular points). 
} 
\smallskip

{\footnotesize
\noindent
\textbf{Acknowlegments}. We are grateful to J.J.\,Nu\~no\,Ballesteros who communicated to us reference\,\cite{NS}. 
M. Kazarian appreciates the support of Russian Science Foundation, project 16-11-10316. 
R. Uribe-Vargas is grateful to Laboratory Solomon Lefschetz UMI2001 CNRS, 
  Universidad Nacional Autonoma de M\'exico.
}

%%%%%%%%%%%%%%%%%%%%%%%%%%%%%%%%%%%%%%%%%%%%%%%%%%%%%%%%%%%%%%%%%%%%%%%%%%%%%%%%%%%%%%%%%%%%%%%%%%%%%%%%%%%%%%%%%%%%%%%%%%
\input{index0}
%%%%%%%%%%%%%%%%%%%%%%%%%%%%%%%%%%%%%%%%%%%%%%%%%%%%%%%%%%%%%%%%%%%%%%%%%%%%%%%%%%%%%%%%%%%%%%%%%%%%%%%%%%%%%%%%%%%%%%%%%%

%%%%%%%%%%%%%%%%%%%%%%%%%%%%%%%%%%%%%%%%%%%%%%%%%%%%%%%%%%%%%%%%%%%%%%%%%%%%%%%%%%%%%%%%%%%%%%%%%%%%%%%%%%%%%%%%%%%
\section{Fundamental cubic form}\label{section-FCF}
%%%%%%%%%%%%%%%%%%%%%%%%%%%%%%%%%%%%%%%%%%%%%%%%%%%%%%%%%%%%%%%%%%%%%%%%%%%%%%%%%%%%%%%%%%%%%%%%%%%%%%%%%%%%%%%%%%%
In this section we introduce an important local projective invariant of a
surface, the so called fundamental cubic form. The second fundamental
form describes the quadratic deviation of a surface from its tangent
plane. In the same way, the fundamental cubic form describes the cubic
deviation of the surface from its quadratic part.

Let $S$ be a smooth surface generically embedded to the affine $3$-space.
Assume that it is given in Monge form $z=f(x,y)$. 
%We can use $x,y$ as local coordinates on the surface.
Consider the Taylor expansion of $f$,
\[f(x,y)=Q(x,y)+C(x,y)+\dots,\]
where $Q$ is homogeneous of degree~$2$ and $C$ of degree~$3$. 
Observe that linear changes of the coordinates~$x$ and~$y$
preserve the homogeneous forms regarded as forms on the tangent plane. We discuss the ambiguity in the
definition of these forms for more general affine changes of coordinates in the ambient space.
\medskip

\noindent
\textit{\textbf{\small Second Fundamental Form}}. 
It is easy to see that $Q$ provides a correctly defined quadratic form on the tangent plane 
taking values on the ``normal line'' $\nu_pS=T_p\RP^3/T_pS$. It is called \textit{second fundamental form} 
of the surface. However, if we treat this form as taking values in real numbers, it is defined up to a factor only. It is sign definite, indefinite or degenerate if the surface is respectively elliptic, hyperbolic or parabolic at the corresponding point.
\medskip

The action of affine changes on the cubic part of the Taylor expansion is more complicated.
Namely, the cubic term $C$ provides a cubic form on the tangent plane up to a factor and up to a
summand of the form $QL$ where $L$ is linear. Theorem\,\ref{theorem:FCF-expression} below claims that, given 
$Q$, the family of cubic forms $C+QL$ for all possible choices of $L$ has a canonical 
representative that we call \emph{fundamental cubic form}. 
This means that the quadratic form $Q$ produces a splitting of the space of cubic forms in two parts: 
\medskip

\noindent 
\textbf{Canonical splitting of the space of cubic forms}. 
Given a nondegenerate quadratic form $Q$ on an abstract $2$-dimensional vector space $V\approx\R^2$,   
the ``convolution with the inverse of $Q$'' %(which is a form $Q^*$ on the dual space~$V^*$) 
is an operation on the homogeneous forms which lowers the degree by two. 
In coordinates, if $Q=a\,x^2+2b\,x\,y+c\,y^2$, this operation %convolution with the inverse of $Q$ 
acts as the second order differential operator
\begin{equation}\label{Qstar}
\Lambda:=c\pd{^2}{x^2}-2b\pd{^2}{x\D y}+a\pd{^2}{y^2}\,.
\end{equation}

Observe that $\left(\frac{1}{4(ac-b^2)}\right)\Lambda\,Q=1$ and that 
this ``convolution operation'' determines the surjective linear map $\,\l_Q:\Sym^3V^*\to\Sym^1V^*\,$ given by 
$\,\l_Q(C)=\Lambda\,C$.  
\medskip

%\begin{proposition*}
\noindent 
\textbf{Splitting Lemma}. 
{\itshape
The form $Q$ determines a canonical splitting of the $4$-dimensional 
space of cubic forms as a direct sum of two $2$-dimensional subspaces, 
\begin{equation}\label{split}
\Sym^3V^*=U_Q^+\oplus U_Q^-\,,
\end{equation}
where $U_Q^+$ consists of the cubic forms multiples of $Q$ {\rm (i.e. written as $QL$ 
with $L$ linear)} and $U_Q^-$ is the kernel of the surjective linear map
$\l_Q:\Sym^3V^*\to\Sym^1V^*$.}
%\end{proposition*}

\begin{proof}
To prove that $U_Q^+$ and $U_Q^-$ are transversal, it is sufficient to check this fact for $Q$ in a normal
form, one for the hyperbolic case and one for the elliptic one.

In the hyperbolic case, we choose coordinites such that $Q=x\,y$, then $U_Q^+$ is spanned
by $x^2y$ and $x\,y^2$, and $U_Q^-$ is spanned by $x^3$ and $y^3$.

In the elliptic case, we choose coordinates such that $Q=\pm(x^2+y^2)$.
Then $U_Q^+$ is spanned by $x^3+x\,y^2$ and $x^2y+y^3$ while $U_Q^-$ is spanned by $x^3-3x\,y^2$ and $3\,x^2y-y^3$.

Therefore the splitting~\eqref{split} holds true for any nondegenerate quadratic form $Q$. 
\end{proof}

Now we assume $V$ is the tangent space to the surface at a non-parabolic point\,: 
\smallskip

\noindent
\textit{\textbf{\small Fundamental Cubic Form}}. 
\emph{The fundamental cubic form} $W$ of a surface (FCF) is the homogeneous
degree 3 form (on the tangent plane $V$) obtained as the projection of the form $C$ 
to the space $U^-_Q$ along the space $U^+_Q$.
\smallskip

This definition leads to the following explicit formula:

\begin{theorem}[proved below]\label{theorem:FCF-expression}
If $f(x,y)=Q(x,y)+C(x,y)+\ldots$ with $Q$ non degenerate, the FCF is given by 
\begin{equation}\label{Formula1-W}
\what{W}=C-Q\,\frac{dH}{4H}\,,
\end{equation}
where $H(x,y)=f_{20}(x,y)f_{02}(x,y)-f_{11}^2(x,y)$ and $dH$ is the linear part of $H$.
\end{theorem}

In the parabolic case (i.e. when $Q$ is degenerate) there is no splitting because
in that case $U_Q^-=U_Q^+$. {\footnotesize 
[For example, if $Q=cy^2$, then $\Lambda=c\D_x^2$ so that $U_Q^-$ is spanned by 
$\{x\,y^2, y^3\}$, coinciding thus with $U_Q^+$.]}
It follows that our definition and expression\,\eqref{Formula1-W} are not applicable 
at the parabolic points. 
But since the cubic form is defined up to a factor, we rescale the form \eqref{Formula1-W} to 
\begin{equation}\label{Coord-Formula}
W=4H C-Q dH
\end{equation}
extending it continuously (and unambiguously) to the parabolic points. 
Thus we consider \eqref{Coord-Formula} as the coordinate expression 
of the fundamental cubic form for all points. 

{\footnotesize
\begin{remark*}
In an abstract vector space $V\simeq \R^2$, the canonical representative of the cubic 
forms $C+QL$ for all possible choices of $L$, defined up to a factor,  is the form 
\begin{equation*}%\label{Formula0-W}
W=4H_QC-2\,Q\,\Lambda C\,,
\end{equation*}
where $H_Q=4(ac-b^2)$ is the Hessian of the quadratic form $Q$.
\end{remark*}
}

\noindent
\textit{\textbf{\small Projective Umbilics} or \textbf{\small Nodes}}. 
  A non parabolic point of a generic surface (i.e., $Q$ is nondegenerate) is called
  \textit{projective umbilic} or \emph{node} (\textit{ellipnode} or \textit{hyperbonode}) if
  the cubic form~$C$ is divisible by $Q$, that is, $C=Q L$ for some
linear function~$L$.

\begin{theorem}\label{thfund3} \ 
Let $W$ be the fundamental cubic form of a smooth surface. 
\begin{enumerate}[ parsep=0.124cm, itemsep=0.0cm,topsep=0.2cm]\itemsep=0.0cm  
\item[{\rm 1.}] The lines of zeroes of the form~$W$ are well defined tangent lines on the surface;

\item[{\rm 2.}] The form $W$ vanishes at the projective umbilics (ellipnodes or hyperbonodes); 

\item[{\rm 3.}] At every elliptic point which is not an ellipnode, the form $W$ has three distinct real lines of zeroes; 

\item[{\rm 4.}] At every hyperbolic point which is not a hyperbonode, the form $W$ has one real line of zeroes; 

\item[{\rm 5.}] At the parabolic points different from godrons the form $W$ has a double zero line which is also
  a double zero line of~$Q$, and a simple zero line tangent to the parabolic curve; 

\item[{\rm 6.}] At every godron the form $W$ has the triple zero line tangent to the parabolic curve.
\end{enumerate}
\end{theorem}

\noindent
\begin{proof} Item~1 follows because the form $W$ is well defined up to a factor. 

Item~2 follows because at the projective umbilics the form $C$ lies, by definition, in the kernel 
of the projection  to $U_Q^-$ along $U_Q^+$.

At a parabolic point we have $H=0$ so that $W=Q\,dH$, which implies immediately the statements of
item~1 as well as of items~5 and~6. 

To prove items 3 and 4, it is sufficient to check them for $Q$ in a normal form: 

In the elliptic case, we choose coordinates such that $Q=\pm(x^2+y^2)$.
Then $U_Q^-$ is spanned by $x^3-3x\,y^2$ and $3\,x^2y-y^3$ and 
assertion 3 follows from the fact that any nonzero combination of the forms $x^3-3x\,y^2$ and
$3\,x^2y-y^3$ has three real zero lines.

In the hyperbolic case, we choose coordinates such that $Q=x\,y$. 
Then $U_Q^-$ is spanned by $x^3$ and $y^3$ and assertion 4 follows from the fact that any 
nonzero combination of the forms $x^3$ and $y^3$ has one real zero line.
Theorem~\ref{thfund3} is proved.
\end{proof}

%{\footnotesize
\begin{remark*}
The fundamental cubic form of a surface is actually a known and well studied object 
in affine differential geometry, see, e.g.,~\cite{NS}. It is defined by
\begin{equation}\label{aff-cubic-form}
  \mathfrak{C}(X,Y,Z)=(\nabla_X h)(Y,Z)\,,
\end{equation}
where $\nabla$ is the connection on the (tangent bundle of the) surface defined by its canonical 
Blaschke structure associated with the affine embedding, and $h$ is the quadratic fundamental form 
with a normalisation that differ from our form $Q$ by a factor, see details in~\cite{NS}. 
Namely, for a surface of the form $z=f(x,y)$ we have explicitly
\begin{align*}
h&=H^{-1/4}Q,\\
\mathfrak{C}&=H^{-1/4}\left(C-\frac14 Q d\log H \right).
\end{align*}

We see the forms $h$ and $\mathfrak{C}$ agree with our respective forms $Q$ and $W$ up to a factor.
The statements 2--4 of Theorem~3 are also known, see~\cite[Sect. II.11]{NS}. 

Observe, however, that our approach to the definition of the 
fundamental cubic form is completely different from that one of affine differential geometry. 
It is important to notice also that, up to a factor, \emph{the quadratic and the cubic fundamental forms 
are well defined local invariants of the surface which come from the projective structure of 
the ambient space}, rather than the affine one.  
Besides, the asymptotic behavior of the field of zeroes of the cubic form on the parabolic line 
subject to statements 5--6 of Theorem\,\ref{thfund3} has not been studied before, to our knowledge. 
In fact, the definition of the cubic form~\eqref{aff-cubic-form} assumes the nondegeneracy of $h$ and 
is not applicable to parabolic points, in contrast to our definition of the (normalised) form~$W$.
\end{remark*}
%}

%%%%%%%%%%%%%%%%%%%%%%%%%%%%%%%%%%%%%%%%%%%%%%%%%%%%%%%%%%%%%%%%%%%%%%%%%%%%%%%%%%%%%%%%%%%%%%%%%%%%%%%%%%%%%%%%%%%%%%%%%
\subsection{Proof of Theorem\,\ref{theorem:FCF-expression}}
%%%%%%%%%%%%%%%%%%%%%%%%%%%%%%%%%%%%%%%%%%%%%%%%%%%%%%%%%%%%%%%%%%%%%%%%%%%%%%%%%%%%%%%%%%%%%%%%%%%%%%%%%%%%%%%%%%%%%%%%%

Notice that if $f(x,y)=Q(x,y)+C(x,y)+\ldots$ 
and $H_f$ denotes its Hessian function 
\[H_f(x,y)=f_{20}(x,y)f_{02}(x,y)-f_{11}^2(x,y)\,,\]
then at the origin we have $H_f(0,0)=H_Q=4(ac-b^2)$; 
and its differential depends only on $Q$ and $C$\,: $\,dH_f(0,0)=dH_{Q+C}$. 
Therefore in the following lemma, which characterises the cubic forms divisible by $Q$, 
we consider the functions $f(x,y)=Q(x,y)+C(x, y)$. 

% The Hessian is positive at the elliptic points of the surface and negative at the hyperbolic points, 
% so that the parabolic curve is given by the equation $H=0$. 

\begin{lemma}\label{Lemma:dH=4HL}
If $f(x,y)=Q(x,y)+C(x, y)$ where $Q$ is a non degenerate quadratic form 
and $C=Q L$ with $L$ linear (i.e. the cubic form~$C$ is divisible by $Q$) then at the origin 
\begin{equation}\label{dH=4HL}
dH_{Q+QL}=4H_QL\,,
\end{equation}
where $dH_{Q+QL}$ is the linear part of the function~$H_{Q+QL}$ and $H_Q=4(ac-b^2)$ is the Hessian of 
the non degenerate quadratic form $Q$.
\end{lemma}

\begin{proof}
If $f=Q+QL$ with $Q(x,y)=ax^2+2bxy+cy^2$ and $L(x,y)=ux+vy$, then 
\[H_f(x,y)=\left(2a(L+1)+2uQ_x\right)\left(2c(L+1)+2vQ_y\right)-\left(2b(L+1)+vQ_x+uQ_y\right)^2\,.\]
Since $L=Q_x=Q_y=0$ at the origin, the linear part of the Hessian at the origin is given by 
$dH_f=4^2(ac-b^2)ux+4^2(ac-b^2)vy$, that is $dH_f=4H_QL$. 
\end{proof}

Lemma\,\ref{Lemma:dH=4HL} suggests the canonical representative $W$ of the cubic 
forms $C+QL$ for all possible choices of $L$ should satisfy $dH_{Q+W}=0$. 
A simple calculation proves the %relation of $\Lambda$ with the Hessian of $Q+C$:  

\begin{lemma}\label{lemma:2LC=dH} 
Let $f=Q+\what{C}$ with $Q$ non degenerate. %where
%\begin{align*}
%Q&=a\,x^2/2+b\,x\,y+c\,y^2/2,\\
%C&=\frac{f_{30}}{3!}x^3+\frac{f_{21}}{2}x^2y+\frac{f_{12}}{2}x\,y^2+\frac{f_{03}}{3!}y^3\,.
%\end{align*}
Then at the origin we have
\begin{equation}\label{relation-LQ=dH}
\Lambda\what{C}=\frac{1}{2}dH_{Q+\what{C}}\,, 
\end{equation}
where $\Lambda$ is the operator \eqref{Qstar} determined by $Q$. 
%that is, at the origin $dH_{Q+C}=2(af_{12}+cf_{30}-2bf_{21})x+2(af_{03}+cf_{21}-2bf_{12})y$. 
\end{lemma} 

%%%%%%%%%%%%%%%%%%%%%%%%%%%%%%%%  Proof of Theorem  %%%%%%%%%%%%%%%%%%%%%%%%%%%%%%%%%%%%

\begin{proof}[\textbf{Proof of Theorem\,{\rm\ref{theorem:FCF-expression}}}]
From the Lemmas\,\ref{Lemma:dH=4HL} and \ref{lemma:2LC=dH} we get $(1/4H)QdH=(1/2H_Q)\,Q\,\Lambda C$. 
Then the linearity of the map $\l_Q$ applied to $\what{W}$ (see \eqref{Formula1-W}) provides the equality 
\begin{equation}\label{linearity-L_Q}
\Lambda\left(\,C-(1/2H_Q)\,Q\,\Lambda C\,\right)=\Lambda C-(1/2H_Q)\,\Lambda(Q\,\Lambda C)\,.
\end{equation} 
So relation\,\eqref{relation-LQ=dH} applied to $\what{C}=Q\,\Lambda C$ and then relation\,\eqref{dH=4HL} applied to $L=\Lambda C$ provide
\begin{equation}\label{relation-L_Q(QL_QC)=2H_QL_QC}
(1/2H_Q)\,\Lambda(Q\,\Lambda C)=(1/2H_Q)\frac{1}{2}dH_{Q+Q\,\Lambda C}=(1/4H_Q)4H_Q\,\Lambda C\,.
\end{equation}

Finally, equalities \eqref{linearity-L_Q} and \eqref{relation-L_Q(QL_QC)=2H_QL_QC} 
imply that $\what{W}$ is annihilated by the operator $\Lambda$. 

The definition of the FCF implies that a change of $C$ by a 
summand of the form $QL$ (i.e., lying in $U^+$) does not change its image under the projection. 
This proves Theorem\,\ref{theorem:FCF-expression}, in the non-parabolic case, 
because the splitting~\eqref{split} is defined intrinsically.
\end{proof}

%%%%%%%%%%%%%%%%%%%%%%%%%%%%%%%%%%%%%%%%%%%%%%%%%%%%%%%%%%%%%%%%%%%%%%%%%%%%%%%%%%%%%%%%%%%%%%%%%%%%%%%%%%%%%%%%%%%%%%%%%
\section{Poincar\'e-Hopf theorem for multivalued line fields}\label{Section-Poincare-Hopf}
%%%%%%%%%%%%%%%%%%%%%%%%%%%%%%%%%%%%%%%%%%%%%%%%%%%%%%%%%%%%%%%%%%%%%%%%%%%%%%%%%%%%%%%%%%%%%%%%%%%%%%%%%%%%%%%%%%%%%%%%%

The classical Poicar\'e-Hopf theorem claims that the sum of indices of singular points of a vector field~$v$
on a compact smooth $n$-dimensional manifold equals the Euler characteristic of that manifold. This equality
holds true also for a manifold with boundary if the field is transverse to the boundary and directed, say,
outside the manifold at every its boundary point. The index is defined as the degree of the mapping
$S^{n-1}\to S^{n-1}$ where the source sphere is the boundary of a small ball on the manifold centred at a
singular point of the field, and the mapping is given by $x\mapsto v(x)/\|v(x)\|$. In the case of a surface ($n=2$)
the index can be treated also as the number of rotations of the vector $v(x)$ about the origin while
the point $x$ makes one turn rotation around the singular point of the field.

In this section, we extend, following the idea of \cite{Kazarian}, the statement of the Poincar\'e-Hopf 
theorem to the case of ``multivalued fields of lines'' on a surface.
\medskip

\noindent 
\textit{\textbf{\small k-Valued Line Fields}}. A \emph{$k$-valued line field} $\t$ on a surface is a correspondence that
associates to a point on a surface an unordered $k$-tuple of pairwise distinct non oriented tangent lines.
We assume that this $k$-tuple of tangent lines depends continuously on the point of the surface and is
defined for all but finitely many points on the surface.
We refer to the points where the $k$-valued line field is not defined as
the \emph{singular} points of this field.

If a point of the surface follows a loop, the continuity of the multivalued line field along this
loop leads to a (cyclic) permutation of its lines. 
\medskip

\noindent
\textit{\textbf{\small Fractional Index}}. 
The \textit{fractional index} of a singular point of a $k$-valued line field is the rational number $p/q$
such that each line of the field comes to its initial position with the same orientation and makes $p$ turns
of rotations while the point of the base makes $q$ turns around the singular point of the surface in positive
direction. It is an element of $(1/2k) \Z$. 
\medskip 

In order to apply this definition, one needs to fix a choice of the orientation of the surface in a neighbourhood
of a singular point of the field. However, the actual value of the fractional index is independent of this choice
and the equality holds independently of whether the surface is orientable or not. Indeed, a change of the orientation
of the surface changes orientations of both the source and the target circles of the mapping defining the index, 
thus preserving its value.

%\begin{theorem}
%The fractional index of an ellipnode of a generic surface can only take the values $1/3$ or $-1/3$. 
%\end{theorem}

\begin{proposition}\label{indices-Euler}
For any closed surface and a multivalued field of lines on it, the sum of fractional indices of the singular
points of the field is equal to the Euler characteristic of the surface. 
\end{proposition}

This extension of the Poincar\'e-Hopf theorem can be proved, for example, by passing to a suitable ramified
covering surface such that the field becomes uni-valued and oriented on the covering surface, and by applying
the usual Poincar\'e-Hopf theorem to the covering surface.

\begin{remark*}
Even for a single-valued line field Proposition\,\ref{indices-Euler} is not formally equivalent to  
Poincar\'e-Hopf theorem if the field is not oriented; the index being half-integer.
\end{remark*}

%\noindent
%\textbf{\small Positive-Negative}. An ellipnode and hyperbonode is said to be \textit{positive} if its
%index is positive, and negative otherwise. 
%\medskip

\noindent 
\textit{\textbf{\small Triviality condition for surfaces with boundary}}. 
The statement of Proposition\,\ref{indices-Euler} can be extended also to the case of a surface with boundary.
We say that a multivalued direction field on a surface is \textit{trivialised along the boundary} if the boundary
contains no singular point and either the directions of the field are never tangent
to the boundary or if at each point of the boundary one of the branches of the field is tangent to the boundary.
Then the equality of Proposition\,\ref{indices-Euler} holds as well.

%%%%%%%%%%%%%%%%%%%%%%%%%%%%%%%%%%%%%%%%%%%%%%%%%%%%%%%%%%%%%%%%%%%%%%%%%%%%%%%%%%%%%%%%%%%%%%%%%%
%%%%%%%%%%%%%%%%%%%%%%%%%%%%%%%%%%%%%%%%%%%%%%%%%%%%%%%%%%%%%%%%%%%%%%%%%%%%%%%%%%%%%%%%%%%%%%%%%%
\section{Proof of Theorem\,\ref{main-theorem}}\label{Section-Proof-Theorem}
%%%%%%%%%%%%%%%%%%%%%%%%%%%%%%%%%%%%%%%%%%%%%%%%%%%%%%%%%%%%%%%%%%%%%%%%%%%%%%%%%%%%%%%%%%%%%%%%%%
%%%%%%%%%%%%%%%%%%%%%%%%%%%%%%%%%%%%%%%%%%%%%%%%%%%%%%%%%%%%%%%%%%%%%%%%%%%%%%%%%%%%%%%%%%%%%%%%%%

\noindent
\textbf{Proof of Theorem\,\ref{main-theorem}a}. % and \ref{main-theorem}b}. 
Let~$S$ be a compact surface in the projective $3$-space and~$H$ be one
of the connected components of its hyperbolic domain. Take for~$\t$
the field of zeroes of the cubic fundamental form. By Theorem~\ref{thfund3},
we have $k=1$. 

The line field~$\t$ extends continuously to the boundary $\D H$.
The extended field is not transverse to the boundary, but the triviality 
condition formulated above is satisfied since the extended field is tangent 
to the parabolic curve~$\D H$ at \emph{each} of its points (including the godrons).
The internal singular points of~$\t$ are exactly the hyperbonodes and the local
computations (Proposition\,\ref{local-index-h} below) show that
\[\ind_h\t=\pm 1\,.\]
We obtain immediately the equality
%\begin{equation}\label{fist-equality}
\[
\,\# h(H)=\chi(H)\,. \eqno{\square}
\]
%\end{equation}

\noindent 
\textbf{Fractional index of singular points on the boundary}. 
Let us also extend the equality of Proposition\,\ref{indices-Euler} to the case when
the multivalued field of directions has singular points on the boundary.
Assume the triviality condition of the multivalued line field holds along the boundary,
except at a finite number of its points.
Take one of these points and pick a small disk centred at this point. Write $\gamma$ for the part of
the boundary of the disk lying in the surface, and $A$ and~$B$ for the endpoints of $\gamma$.  
Since $A$ and~$B$ belong to the part of the boundary of the surface where the field is trivialised, 
there is a homeomorphism of the tangent planes at the points $A$ and $B$ that identifies both values
of the direction field and the (cooriented) tangent direction of the boundary.
This identification provides a multivalued direction field along the closed path
$\g/(A\sim B)$.  
\medskip

\noindent
\textit{\textbf{\small Border Fractional Index}}. 
Define the \textit{fractional index of the multivalued field of lines at a singular point on the boundary}
as the index of the obtained closed path $\g/(A\sim B)$.
\smallskip

For example, using this definition, the fractioanl indices of the $2$-valued fields of lines 
of Fig.\,\ref{index-at-boundary} along the closed paths $\g/(A\sim B)$ are respectively $1/2$ and $-1/2$.

\begin{figure}[h] %< préférence de placement h , t , b or p >
\centering 
\includegraphics[scale=0.32]{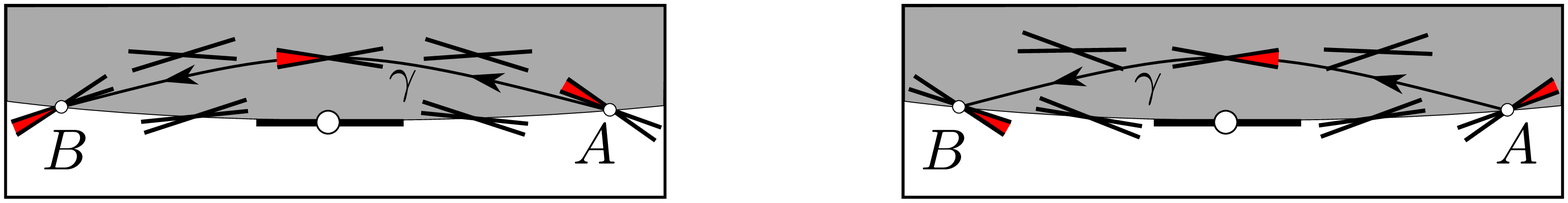}
\caption{\small Boundary singular points with fractional indices $1/2$ (left) and $-1/2$ (right).}
\label{index-at-boundary}
\end{figure}

This definition is justified by the following theorem (with its proof). 

\begin{theorem}\label{Euler-extended}
For any multivalued direction field on a surface with boundary, the sum of indices of all singular points,
both internal and lying on the boundary, is equal to the Euler characteristic of the surface.
\end{theorem}

\noindent
\emph{Proof}. Let us extend the surface by attaching a collar along the boundary which is a narrow strip
(Fig.\,\ref{attaching-collar}-centre).
The field can be extended to the collar in such a way that it becomes trivialised along the (new) boundary,
without singularities inside the collar, and the boundary singularity of the original field becomes a internal
singularity for the extended field (Fig.\,\ref{attaching-collar}-right). Thus, the situation of Proposition\,\ref{indices-Euler} is applied and leads to
the equality of Theorem\,\ref{Euler-extended}.

\begin{figure}[h] %< préférence de placement h , t , b or p >
\centering
\includegraphics[scale=0.3]{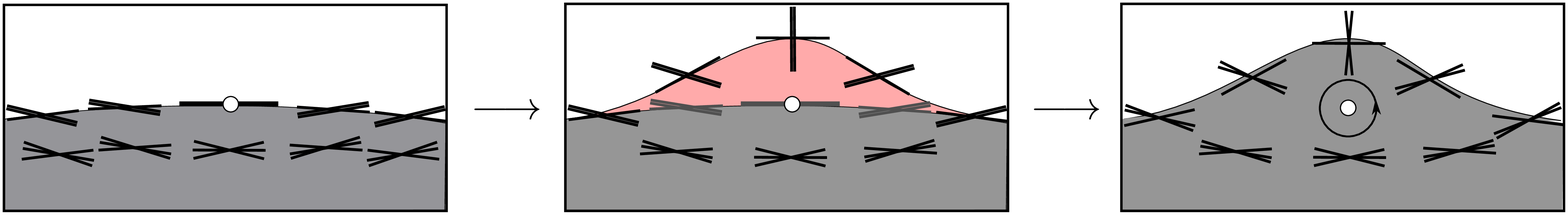}
\caption{\small Computation of the index at a boundary singular point.}
\label{attaching-collar}
\end{figure}

\begin{remark*}
The statement of Theorem\,\ref{Euler-extended} is applicable, for example, in the case of a vector field which is not
necessarily transverse to the boundary. A point of the boundary is `singular' for the field if it is an isolated
point where the field is tangent to the boundary. The index of such point is half-integer.
For example, the vector field $v=\partial_x$ on the unit disk $D={x^2+y^2\le 1}$ has no internal singular
points, but it has two boundary singular points of indices $1/2$ both, which gives $\chi(D)=1/2+1/2=1$.
\end{remark*}
\smallskip

\noindent
\textbf{Proof of Theorem\,\ref{main-theorem}b}. 
%The Euler Characteristic of a component of the hyperbolic domain is also given by the godrons of its boundary. 
The only singularities of the $2$-valued field of asymptotic lines, on a connected component 
$H$ of the hyperbolic domain, are the godrons.
%We shall prove that these indices  %, for the field of asymptotic lines, 
%are equal to $1/2$ or to $-1/2$. 

Consider the normal form $(\mathcal{P})$ near a godron: $f(x,y)=y^2/2-x^2y+\rho x^4/2$. 
The asymptotic lines are the zeroes of the second fundamental form 
\[Q=(-y+6\rho x^2)\,dx^2-4x\,dxdy+dy^2\,.\]
\textbf{Fact 1b}. {\itshape Near the godron the asymptotic lines are never parallel to the $y$-axis} 
(because the coefficient of $dy^2$ in $Q$ is constant) and {\itshape the sectors where 
$Q$ takes positive (or negative) values are symmetric with respect to the $y$-axis}: 
$Q_{(x,y)}(v_x,v_y)=Q_{(-x,y)}(-v_x,v_y)$. 

At the parabolic points near a positive godron $g^+$, the half-asymptotic lines directed to the 
hyperbolic domain point towards $g^+$ (see Fig.\,\ref{index-godron}). 
So at a hyperbolic point $A$ near $g^+$ with $x>0$ %(Sect. 2.3, Fig.\,\ref{index-godron}), 
we have two close half-lines of zeroes that determine a thin sector directed to $H$ and to $g^+$  
(Fig.\,\ref{index-at-boundary}-left). Fact\,1b implies that after moving 
on $\g$ to a hyperbolic point $B$ with $x<0$, near $g^+$, the considered sector is directed away from 
$g^+$ and towards the elliptic side. %(see Fig.\,\ref{index-at-boundary}-left). 
Therefore the fractional index of the field at $g^+$ equals $1/2$. 

Exactly in the same way, one proves that the fractional index of the $2$-valued field of asymptotic lines 
at a negative godron equals $-1/2$ (see Fig.\,\ref{index-at-boundary}-right). 

Since the Euler characteristic of $H$ equals the sum of indices at the godrons on the boundary 
$\D H$ (by Theorem\,\ref{Euler-extended}), the algebraic sum of godrons on $\D H$ equals $2\chi(H)$. \fin
\medskip

{\footnotesize
\noindent 
One also can prove the equality $\# g(\D H)=2\chi(H)$ by using the double covering of the asymptotic 
directions (it is explained and used in \cite{Uribegodron}). }
\medskip

\noindent
\textbf{Proof of Theorem\,\ref{main-theorem}c}. 
Let $E$ be one of the connected components of the elliptic domain.
We are going to apply Theorem\,\ref{Euler-extended} to the line field $\t$ of zeroes of the
fundamental cubic form on $E$.
By Theorem~\ref{thfund3}, $\t$ is a $3$-valued line field that has singularities at the ellipnodes
(which are internal points of the elliptic domain) and at the godrons (which are boundary points
of that elliptic domain).
Thus, to obtain the Euler characteristic of $E$, we have to sum the contribution of the indices
of the ellipnodes in $E$, and that of the indices of the godrons lying on the boundary $\D E$. 
\smallskip

%\noindent
%\textbf{Fact}. {\itshape The fractional index of an ellipnode and of a godron {\rm (as boundary point of the elliptic
%    domain)} cannot be a multiple of $1/6$.}
%\smallskip

%\noindent
%\textit{Proof}.
%It follows because the sign of the FCF in the domain between two consecutive rays of the line field  
%is invariant along any path not passing through a singular point.  
%\smallskip

%Therefore, each such fractional index is an integral multiple of $1/3$. 
%\smallskip

\noindent
\textbf{Fact  1c}. The local computations for the ellipnodes (Proposition\,\ref{local-index-e} below) show that
\begin{equation}\label{index-elliptic}
\ind_{e}\t=\pm\frac13\,.
\end{equation}

To find the local index of $\t$ at a godron, notice that the field~$\t$ extends continuously
to the boundary $\D E$, but the extended field does not satisfy the 
triviality condition: two of the directions glue together at the
boundary points and the third one becomes tangent to the boundary.
This degeneracy can be resolved by moving slightly from the
boundary point to a close internal point of~$E$.
Applying a small such modification we obtain a field satisfying the 
necessary conditions on the boundary of~$E$.
%The described modification will not introduce any boundary singular point in a
%neighborhood of a point which is not a godron.
The fractional index of $\t$ at a godron is then computed using the above definition (Fig.\,\ref{index-at-boundary}). 
%However, a godron will produce a singular point on
%the boundary for the modified field whose index can be computed by local
%computations.
\smallskip

\noindent
\textbf{Fact  2c}. The local computations at the godrons on $\D E$ (Proposition\,\ref{local-index-p} below) show that  
\textit{the index at a positive godron equals $-1/3$ and at a negative godron equals $1/3$}\,: 
\begin{equation}\label{indices-godrons}
\ind_{g^+}\t=-\frac13,\qquad\ind_{g^-}\t=\frac13\,.
\end{equation}

Theorem\,\ref{Euler-extended} together with equalities \eqref{index-elliptic} and \eqref{indices-godrons} imply the relation
\[\# e(E)\,-\,\# g(\D E)\,=\,3\chi(E)\,. \eqno{\square}\]

%%%%%%%%%%%%%%%%%%%%%%%%%%%%%%%%%%%%%%%%%%%%%%%%%%%%%%%%%%%%%%%%%%%%%%%%%%%%%%%%%%%%%%%%%%%%%%%%%%
\section{Local indices at hyperbonodes, ellipnodes and godrons}\label{Section-Local-Indices}
%%%%%%%%%%%%%%%%%%%%%%%%%%%%%%%%%%%%%%%%%%%%%%%%%%%%%%%%%%%%%%%%%%%%%%%%%%%%%%%%%%%%%%%%%%%%%%%%%%
{\footnotesize
  We shall compute the local index of the multivalued line field $\t$, defined by the fundamental cubic form
  $W=4HC-QdH$, at generic hyperbonodes, ellipnodes and godrons. To perform these computations (by hand),
  we only need to find the relevant terms of $C$, $Q$, $H$ and $dH$.} 

\paragraph{6.1 Local Index at a Hyperbonode.} Write the surface in Monge form $z=f(x,y)$.

%%%%%%%%%%%%%%%%%%%%%%%%%%%%%%%%%%%%%%%%%%%%%%%%%%%%%%%%%%%%%%%%%%%%%%%%%%%%%%%%%%%%%%%%%%%%%%%%%%%%%%%%%%%%%%%%%
%%%%%%%%%%%%%%%%%%   INDEX HYPERBONODE   %%%%%%%%%%%%%%%%%%%%%%%%%%%%%%%%%%%%%%%%%%%%%%%%%%%%%%%%%%%%%%%%%%%%%%%%
%%%%%%%%%%%%%%%%%%%%%%%%%%%%%%%%%%%%%%%%%%%%%%%%%%%%%%%%%%%%%%%%%%%%%%%%%%%%%%%%%%%%%%%%%%%%%%%%%%%%%%%%%%%%%%%%%
\begin{proposition}\label{local-index-h}
The local index of a hyperbonode $h$, taking the asymptotic lines as
coordinate axes, equals
\[\ind_h(\t)=1\cdot\sign\left(4f_{11}^2f_{40}f_{04}-(2f_{11}f_{31}-3f_{21}^2)(2f_{11}f_{13}-3f_{12}^2)\right)\,.\]  
\end{proposition}

\begin{corollary*}
For the above normal forms of Landis-Platonova and of Ovsienko-Tabachnikov we get the respective 
expressions of the index
\[\ind_h(\t)=\pm 1-ab \  \qquad \text{and} \ \qquad \ind_h(\t)=IJ\mp 1\,.\]
\end{corollary*}

\begin{remark}\label{remark-second-expression}
  If at $h$ we take the diagonals $\,y=\pm x\,$ as asymptotic lines and we assume the cubic 
  terms of $f$ are missing for the chosen affine coordinate system, then
\begin{equation}\label{second-index-expression}
\ind_h(\t)=1\cdot\sign\left((f_{40}+3f_{22})(f_{04}+3f_{22})-(f_{31}+3f_{13})(f_{13}+3f_{31})\right)\,.
\end{equation}   
\end{remark}

{\footnotesize
\noindent
\textbf{Notation}. In order to get not so long lines along the proofs, we shall 
replace the partial derivatives $f_{ij}$ with constants $\a, u, v, a, b, I, J$, etc, just for the calculations. }
\medskip

\noindent
\textit{Proof of Proposition}\,\ref{local-index-h}. 
  Let us locally express the surface in Monge form
\[f(x,y)=\a xy+\frac{1}{2}ux^2y+\frac{1}{2}vxy^2+\frac{1}{3!}(ax^3y+bxy^3)+\frac{1}{4!}(Ix^4+Jy^4).\]
  The explicit expressions of the relevant terms for $H$, $dH$, $Q$ and $C$ are {\small
\begin{enumerate}[label=\empty, parsep=0.124cm, itemsep=0.0cm,topsep=0.2cm]\itemsep=0.0cm  
\item $H\approx-\a^2-2\a(ux+vy)+\ldots$\,;
\item $dH\approx-(2\a u+2(u^2+\a a)x+uvy+\ldots)dx-(2\a v+2(v^2+\a b)y+uvx+\ldots)dy $\,;
\item $Q \approx\dfrac{1}{2}(uy+axy+\ldots)dx^2+(\a+ux+vy+\ldots)dxdy
  +\dfrac{1}{2}(vx+bxy+\ldots)dy^2$\,;
\item $C=\dfrac{(ay+Ix)dx^3}{3!}+\dfrac{(u+ax)dx^2dy}{2}+\dfrac{(v+by)dxdy^2}{2}+\dfrac{(bx+Jy)dy^3}{3!}$\,.
\end{enumerate}}
Then the expression for $4HC$ up to terms of order $1$ in $x$, $y$ is 
\[4HC\approx-\dfrac{2}{3}\a^2(ay+Ix)dx^3+\f(x,y)dx^2dy+\psi(x,y)dxdy^2-\dfrac{2}{3}\a^2(bx+Jy)dy^3,\] 
where $\f(x,y)=-2\a^2(u+ax)+4\a u(ux+vy)$ and $\psi(x,y)=-2\a^2(v+by)+4\a v(ux+vy)$.

The corresponding expression for $QdH$, up to its first order terms, is 
\[QdH\approx -\a u^2y\,dx^3+\f(x,y)dx^2dy+\psi(x,y)dxdy^2 -\a v^2x\,dy^3\,.\]
Therefore, the fundamental cubic form $W=4HC-QdH$ is given by 
\begin{equation}\label{FCF-hyperbonode}
W\approx-\frac{\a}{3}\left(2\a Ix+(2\a a-3u^2)y\right)dx^3-\frac{\a}{3}\left((2\a b-3v^2)x+2\a Jy\right)dy^3.
\end{equation}

Given a point that makes a positive turn on a very small circle around the origin,  
the line of zeroes of \eqref{FCF-hyperbonode} makes a positive turn if and only if
the image of our small circle by the map
$(x,y)\mapsto \left(2\a Ix+(2\a a-3u^2)y, (2\a b-3v^2)x+2\a Jy\right)$
makes a positive turn around $(0,0)$; that is, if and only if this map
preserves the orientation. Since its determinant is equal to 
$4\a^2 IJ-(2\a a-3u^2)(2\a b-3v^2)$, we get that  
\[\ind_h(\t)=1\cdot\sign\left(4f_{11}^2f_{40}f_{04}-(2f_{11}f_{31}-3f_{21}^2)(2f_{11}f_{13}-3f_{12}^2)\right)\,. \eqno{\square}\]
%%%%%%%%%%%%%%%% End of the proof %%%%%%%%%%%%%%%%%%%%%%

%%%%%%%%%%%%%%%%%%%%%%%%%%%%%%%%%%%%%%%%%%%%%%%%%%%%%%%%%%%%%%%%%%%%%%%%%%%%%%%%%%%%%%%%%%%%%%%%%%%%%%%%%%%%%%%%%
%%%%%%%%%%%%%%%%%%   INDEX ELLIPNODE   %%%%%%%%%%%%%%%%%%%%%%%%%%%%%%%%%%%%%%%%%%%%%%%%%%%%%%%%%%%%%%%%%%%%%%%%%%
%%%%%%%%%%%%%%%%%%%%%%%%%%%%%%%%%%%%%%%%%%%%%%%%%%%%%%%%%%%%%%%%%%%%%%%%%%%%%%%%%%%%%%%%%%%%%%%%%%%%%%%%%%%%%%%%%
\paragraph{6.2 Local Index at an Ellipnode.} Writing the surface in Monge form $z=f(x,y)$ with
$Q=\frac{\a}{2}(x^2+y^2)$ and assuming the cubic terms of $f$ are missing (for the chosen affine 
coordinate system) we get the  

\begin{proposition}\label{local-index-e}
The local index of $\t$ at an ellipnode $e$ for which $Q=\frac{\a}{2}(x^2+y^2)$, equals
\[\ind_e(\t)=\frac{1}{3}\cdot\sign\left((f_{31}-3f_{13})(f_{13}-3f_{31})-(f_{40}-3f_{22})(f_{04}-3f_{22})\right)\,.\] 
{\rm(Compare this formula with expression \eqref{second-index-expression} of Remark\,\ref{remark-second-expression}.)}
\end{proposition}

%\begin{corollary*}
%  For the prenormal form {\rm ($\mathcal{E}$)} the index equals
%  \[\ind_e(\t)=\frac{1}{3}\cdot\sign\left((a-3b)(b-3a)-(I-3c)(J-3c)\right)\,.\]
%\end{corollary*}

\noindent
\textit{Proof of Proposition}\,\ref{local-index-e}. 
  Let us locally express the surface in Monge form
\[f(x,y)=\frac{\a}{2}(x^2+y^2)+\frac{1}{3!}(ax^3y+bxy^3)+\frac{c}{4}x^2y^2+\frac{1}{4!}(Ix^4+Jy^4).\]
  The explicit expressions of the relevant terms for $H$, $dH$, $Q$ and $C$ are {\small
\begin{enumerate}[label=\empty, parsep=0.124cm, itemsep=0.0cm,topsep=0.2cm]\itemsep=0.0cm  
\item $H\approx \a^2+\ldots$\,;
\item $dH\approx\a\left((a+b)y+(c+I)x+\ldots\right)dx+\a\left((a+b)x+(c+J)y+\ldots\right)dy $\,;
\item $Q \approx\dfrac{1}{2}(\a+\ldots)dx^2+(\ldots)dxdy+\dfrac{1}{2}(\a+\ldots)dy^2$\,;
\item $C=\dfrac{(ay+Ix)dx^3}{3!}+\dfrac{(ax+cy)dx^2dy}{2}+\dfrac{(by+cx)dxdy^2}{2}+\dfrac{(bx+Jy)dy^3}{3!}$\,.
\end{enumerate}}
Then the expression for $4HC$ up to terms of order $1$ in $x$, $y$\, is 
\[4HC\approx 2\a^2\left(\dfrac{1}{3}(ay+Ix)dx^3+(ax+cy)dx^2dy+(by+cx)dxdy^2+\dfrac{1}{3}(bx+Jy)dy^3\right),\] 
and the corresponding expression for $QdH$, up to its first order terms, is 
\[QdH\approx \frac{\a^2}{2}\left(\left((a+b)y+(c+I)x\right)(dx^3+dxdy^2)+\left((a+b)x+(c+J)y\right)(dx^2dy+dy^3)\right)\,.\]
Therefore, the fundamental cubic form $W=4HC-QdH$ is given by {\small 
\[W\approx\frac{\a^2}{6}\left(\left((I-3c)x+(a-3b)y\right)(dx^3-3dxdy^2)+
\left((3a-b)x+(3c-J)y\right)(3dx^2dy-dy^3)\right).\]}

Hence, the local index of the ellipnode $e$ is $1/3$ multiplied by the sign of the determinant of the linear map
$(x,y)\mapsto \left((I-3c)x+(a-3b)y,(3a-b)x+(3c-J)y \right)$, which is equal to
$(a-3b)(b-3a)-(I-3c)(J-3c)$. Therefore  
\[\ind(e)=\frac{1}{3}\cdot\sign\left((f_{31}-3f_{13})(f_{13}-3f_{31})-(f_{40}-3f_{22})(f_{04}-3f_{22})\right)\,. \eqno{\square}\]
%%%%%%%%%%%%%%%% End of the proof %%%%%%%%%%%%%%%%%%%%%%

%%%%%%%%%%%%%%%%%%%%%%%%%%%%%%%%%%%%%%%%%%%%%%%%%%%%%%%%%%%%%%%%%%%%%%%%%%%%%%%%%%%%%%%%%%%%%%%%%%%%%%%%%%%%%%%%%
%%%%%%%%%%%%%%%%%%   INDEX GODRON   %%%%%%%%%%%%%%%%%%%%%%%%%%%%%%%%%%%%%%%%%%%%%%%%%%%%%%%%%%%%%%%%%%%%%%%%%%%%%
%%%%%%%%%%%%%%%%%%%%%%%%%%%%%%%%%%%%%%%%%%%%%%%%%%%%%%%%%%%%%%%%%%%%%%%%%%%%%%%%%%%%%%%%%%%%%%%%%%%%%%%%%%%%%%%%%
\paragraph{6.3 Local Index at a Godron.} In this case, we base our computations on basic
geometric properties of the three-valued line field $\t$ near the considered godron. 
\begin{proposition}\label{local-index-p}
  At a positive godron the local index of the field $\tau$ equals $-1/3$ and at a negative godron it
  is equal to $1/3$\,$:$
\[\ind_{g^+}(\tau)=-\frac{1}{3} \quad \text{and} \quad \ind_{g^-}(\tau)=\frac{1}{3}\,.\] 
\end{proposition}

\begin{proof}
  Near a godron given in Landis-Platonova normal form 
  \[f(x,y)=\frac{1}{2}y^2-x^2y+\frac{1}{2}\rho x^4\,,\]
  we get the following expression for the fundamental cubic form 
\begin{equation}\label{fcf-godron}
W=\left(-(4\rho+8)y+(12\rho^2-8\rho)x^2\right)xdx^3+6\left(y+\rho x^2\right)dx^2dy-(6\rho x)dxdy^2+dy^3.
\end{equation}

\noindent
\textbf{Basic observations}. 
By formula \eqref{fcf-godron} the lines of zeroes of $W$ are never vertical (never parallel
to the $y$-axis) because the coefficient of $dy^3$ is the constant $1$ (near a godron the three
lines are close to the horizontal). Moreover, these lines of zeroes and the sectors where $W$ takes 
positive (or negative) values are symmetric with respect to the $y$-axis:
  $W_{(x,y)}(v_x,v_y)=W_{(-x,y)}(-v_x,v_y)$. We only need these basic observations from $W$. 

At the parabolic points near a positive godron $g^+$, the half-asymptotic lines directed to the hyperbolic domain
point towards $g^+$ (see Fig.\,\ref{index-godron}). Then
at the elliptic points with $x<0$, near $g^+$, we have two close half-lines of zeroes that determine a thin sector
pointing to $H$ and to $g^+$ (sector $1$ in Fig.\ref{index-g+=-1/3}).
Our basic observations imply that after moving to the elliptic points with $x>0$,
near $g^+$, the considered sector points away from $g^+$ and is adjacent to the parabolic curve from
the elliptic side.
Then we get the index $-1/3$ because the sector makes a complete negative
turn after three such loops around $g^+$. 

\begin{figure}[h] %< préférence de placement h , t , b or p >
\centering
\includegraphics[scale=0.37]{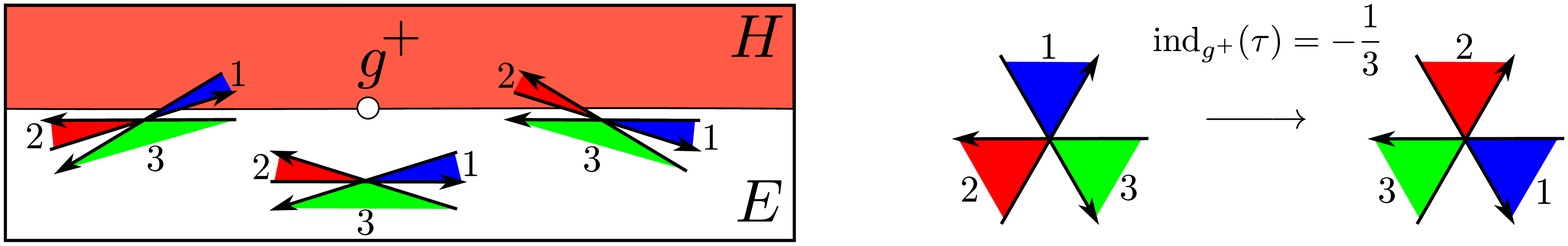}
\caption{\small A positive godron contributes $-1/3$ to the Euler characteristic of $E$.}
\label{index-g+=-1/3}
\end{figure}

At the parabolic points near a negative godron $g^-$, the half-asymptotic lines directed to the hyperbolic domain
point away from  $g^-$ (Fig.\,\ref{index-godron}). Thus 
at the elliptic points with $x<0$, near $g^-$, we have two close half-lines of zeroes that determine a thin sector
pointing to $H$ and away from $g^-$ (sector $1$ in Fig.\ref{index-g-=1/3}).
Our basic observations imply that after moving to the elliptic
points with $x>0$, near $g^-$, the considered sector points to $g^-$ and is adjacent to the parabolic
curve from the elliptic side. Then we get the index $1/3$ because the sector makes a complete positive
turn after three such loops around $g^-$. 
\end{proof}

\begin{figure}[h] %< préférence de placement h , t , b or p >
\centering
\includegraphics[scale=0.37]{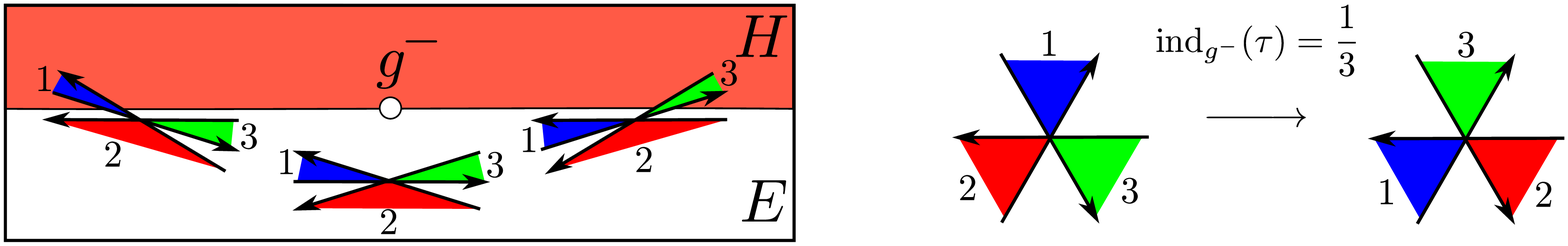}
\caption{\small A negative godron contributes $+1/3$ to the Euler characteristic of $E$.}
\label{index-g-=1/3}
\end{figure}

%%%%%%%%%%%%%%%%%%%%%%%%%%%%%%%%%%%%%%%%%%%%%%%%%%%%%%%%%%%%%%%%%%%%%%%%%%%%%%%%%%%%%%%%%%%%%%%%%%%%%%%%
%%%%%%%%%%%%%%%%%     REFERENCES     %%%%%%%%%%%%%%%%%%%%%%%%%%%%%%%%%%%%%%%%%%%%%%%%%%%%%%%%%%%%%%%%%%%
%%%%%%%%%%%%%%%%%%%%%%%%%%%%%%%%%%%%%%%%%%%%%%%%%%%%%%%%%%%%%%%%%%%%%%%%%%%%%%%%%%%%%%%%%%%%%%%%%%%%%%%%

\end{document}

%% file: index0.tex
%%%%%%%%%%%%%%%%%%%%%%%%%%%%%%%%%%%%%%%%%%%%%%%%%%%%%%%%%%%%%%%%%%%%%%%%%%%%%%%%%%%%%%%%%%%%%%%%%%%%%%%%%%%%%%%%%%%
\section{Properties of surfaces related to the characteristic points}\label{section-properties}
%%%%%%%%%%%%%%%%%%%%%%%%%%%%%%%%%%%%%%%%%%%%%%%%%%%%%%%%%%%%%%%%%%%%%%%%%%%%%%%%%%%%%%%%%%%%%%%%%%%%%%%%%%%%%%%%%%%
{\footnotesize Let us mention some basic features of generic smooth surfaces related to 
the characteristic points.}

%%%%%%%%%%%%%%%%%%%%%%%%%%%%%%%%%%%%%%%%%%%%%%%%%%%%%%%%%%%%%%%%%%%%
\paragraph{2.1 Contact with Lines.}
%%%%%%%%%%%%%%%%%%%%%%%%%%%%%%%%%%%%%%%%%%%%%%%%%%%%%%%%%%%%%%%%%%%%
There are several characterisations of the different kinds of points of smooth surfaces. 
Let us describe the classification of points of generic surfaces according to their contact 
with lines.
The points of a surface in $3$-space are characterised by \textit{asymptotic lines}
(tangent lines of the surface with more than $2$-point contact).
A point is called hyperbolic (resp. elliptic) if there is two distinct asymptotic lines
(resp. no asymptotic line). The \textit{parabolic curve} consists of the points where
there is a unique (but double) asymptotic line. The \textit{flecnodal curve} is the locus of
hyperbolic points where an asymptotic line admits more than $3$-point contact with the surface.

\begin{figure}[ht]
\centering
\includegraphics[width=.6\textwidth]{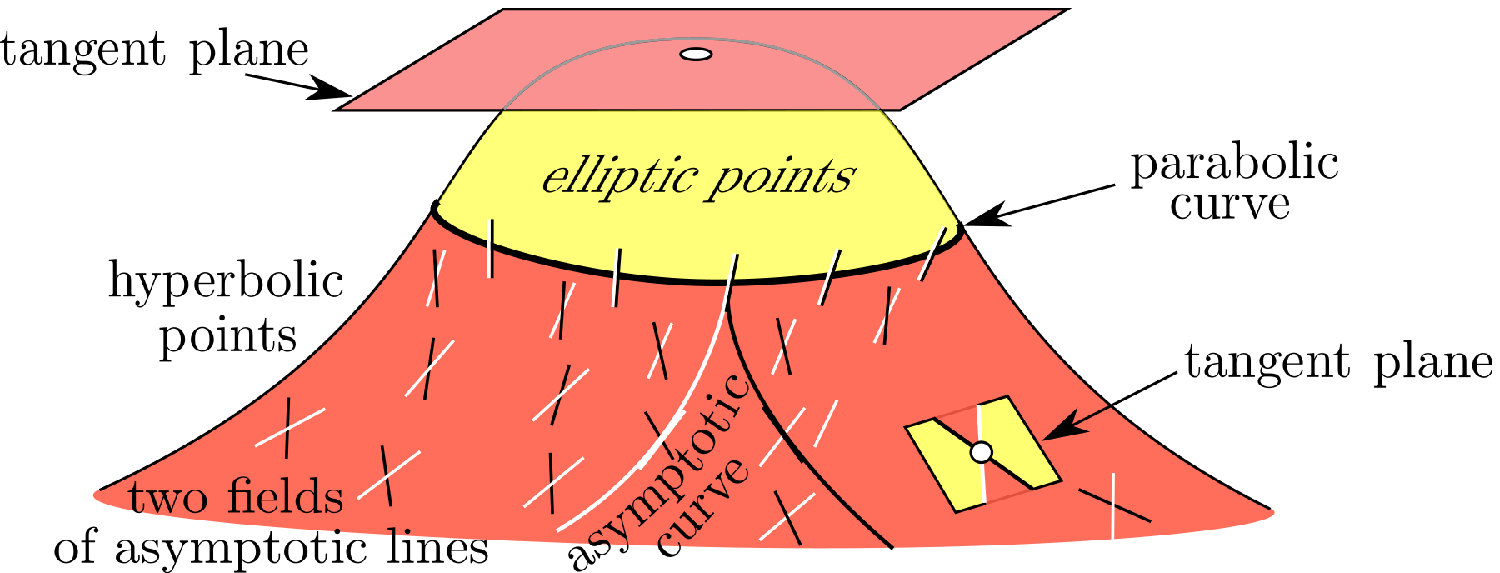}
\end{figure}

%%%%%%%%%%%%%%%%%%%%%%%%%%%%%%%%%%%%%%%%%%%%%%%%%%%%%%%%%%%%%%%%%%%%
\paragraph{2.2 Flecnodal Curve, Hyperbonodes and Godrons.}\label{flecnodal-hyperbonode}
%%%%%%%%%%%%%%%%%%%%%%%%%%%%%%%%%%%%%%%%%%%%%%%%%%%%%%%%%%%%%%%%%%%%
An \textit{asymptotic curve} is an integral curve of a field of asymptotic lines.

Through each hyperbolic point of a surface in oriented space $\RP^3$ there passe two asymptotic
curves; one of them is \textit{left} (the first $3$ derivatives form a negative frame) and the other
is \textit{right} (the frame of the first $3$ derivatives is positive). {\itshape The flecnodal curve consists
of the inflections of the asymptotic curves.} 
The branch formed by the inflections of the left (right) asymptotic curves is called
\textit{left} (\textit{right}) \textit{flecnodal curve}.

{\itshape A hyperbonode is an intersection point of the left and
  right flecnodal curves} (Fig.\,\ref{a-hyperbonode}). 

Moreover, {\itshape a godron is a point of tangency of the flecnodal and parabolic curves. 
It locally separates the left and right branches of the flecnodal curve} (Fig.\,\ref{a-godron}).

\begin{figure}[h]
\centering
\begin{minipage}[t]{5cm}
\centering
\includegraphics [scale=0.35]{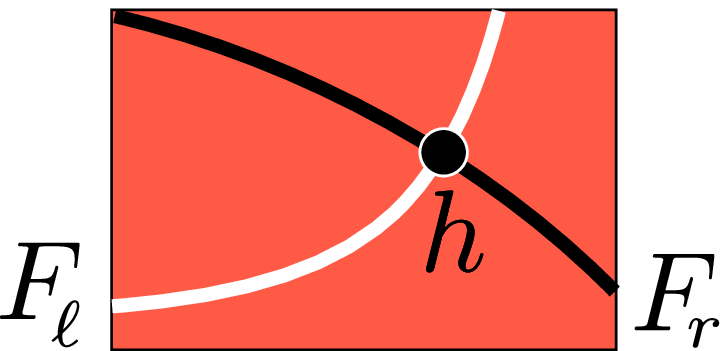}
\caption{\small A hyperbonode.}
\label{a-hyperbonode}
\end{minipage}
\hspace{1cm}
\begin{minipage}[t]{6cm}
\centering
\includegraphics [scale=0.22]{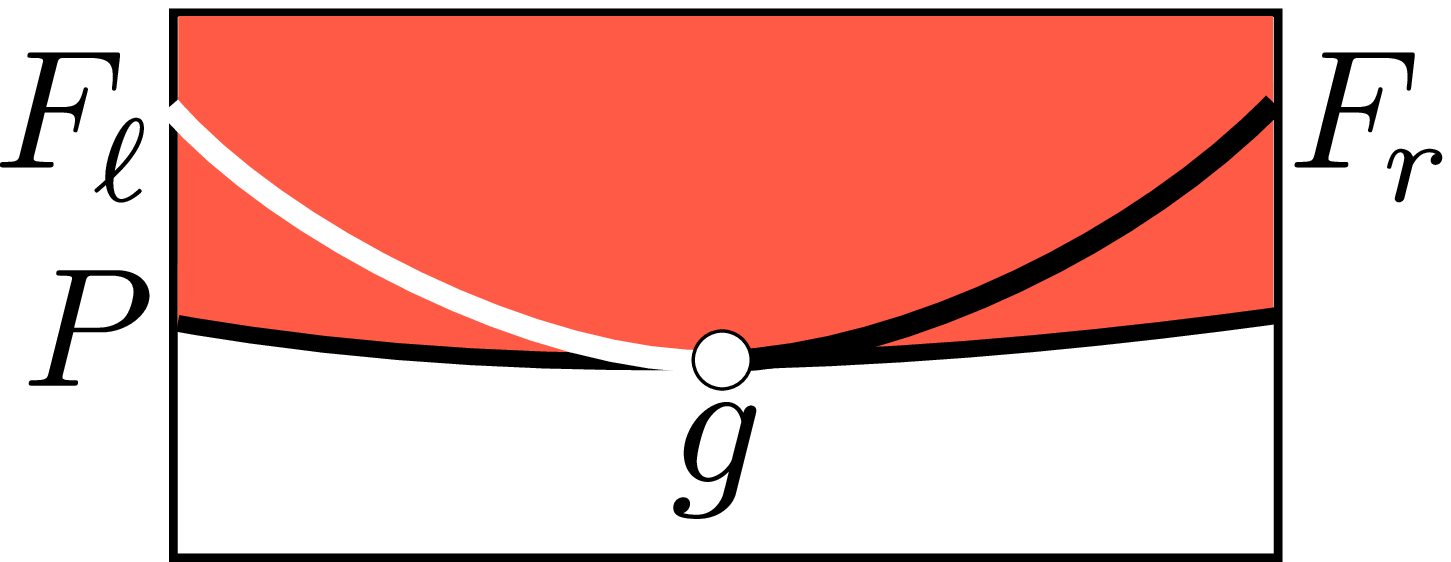}
\caption{\small Flecnodal curve at a godron.}
\label{a-godron}
\end{minipage}
\end{figure}

\begin{example}
  A one sheet hyperboloid is infinitely degenerate\,: the asymptotic lines, at every point,
  are part of the surface. %(having $\infty$-point contact with it).
  Therefore every point is a hyperbonode (Fig.\,\ref{1sheet-hyperboloid}).
\end{example}

\begin{example}
A generic torus (non-symmetric) is a more typical example.   
Its exterior part is elliptic and the interior one is hyperbolic.
The parabolic curve consists of two closed curves that separate the hyperbolic 
domain from the elliptic one (Fig.\,\ref{genric-torus}).
\end{example}

\begin{figure}[h]
\centering
\begin{minipage}[t]{6cm}
\centering
\includegraphics [scale=0.33]{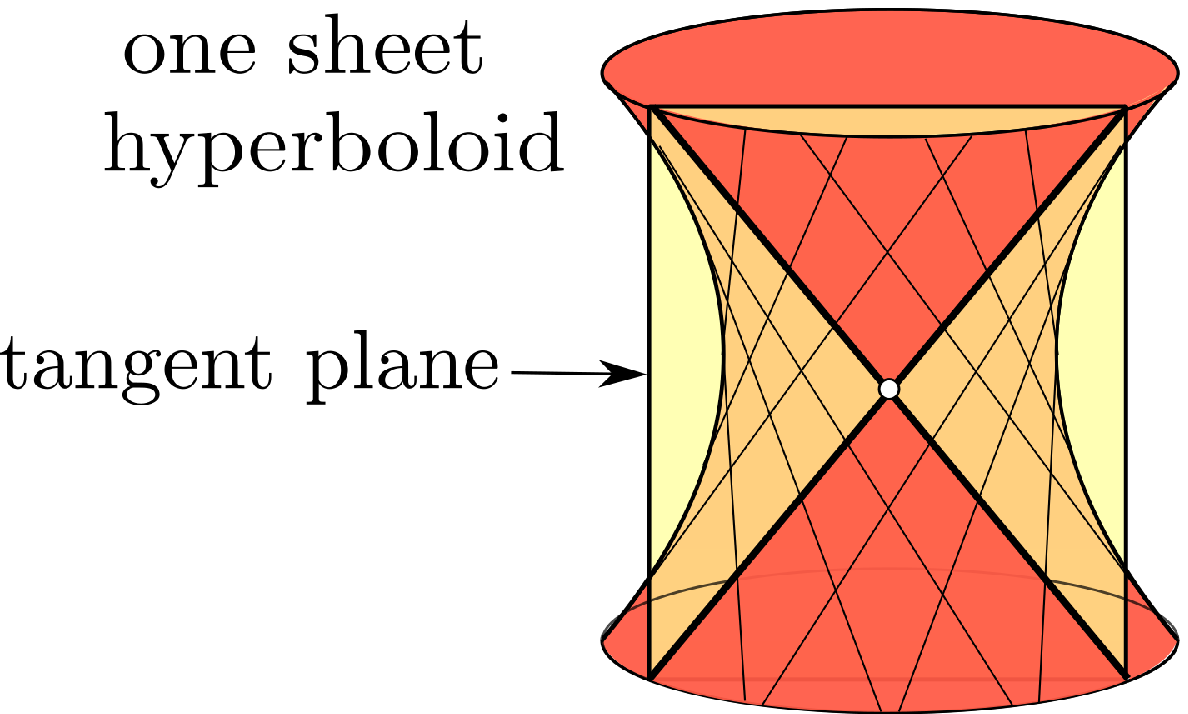}
\caption{\small A one sheet hyperboloid.}
\label{1sheet-hyperboloid}
\end{minipage}
\hspace{1cm}
\begin{minipage}[t]{6cm}
\centering
\includegraphics [scale=0.35]{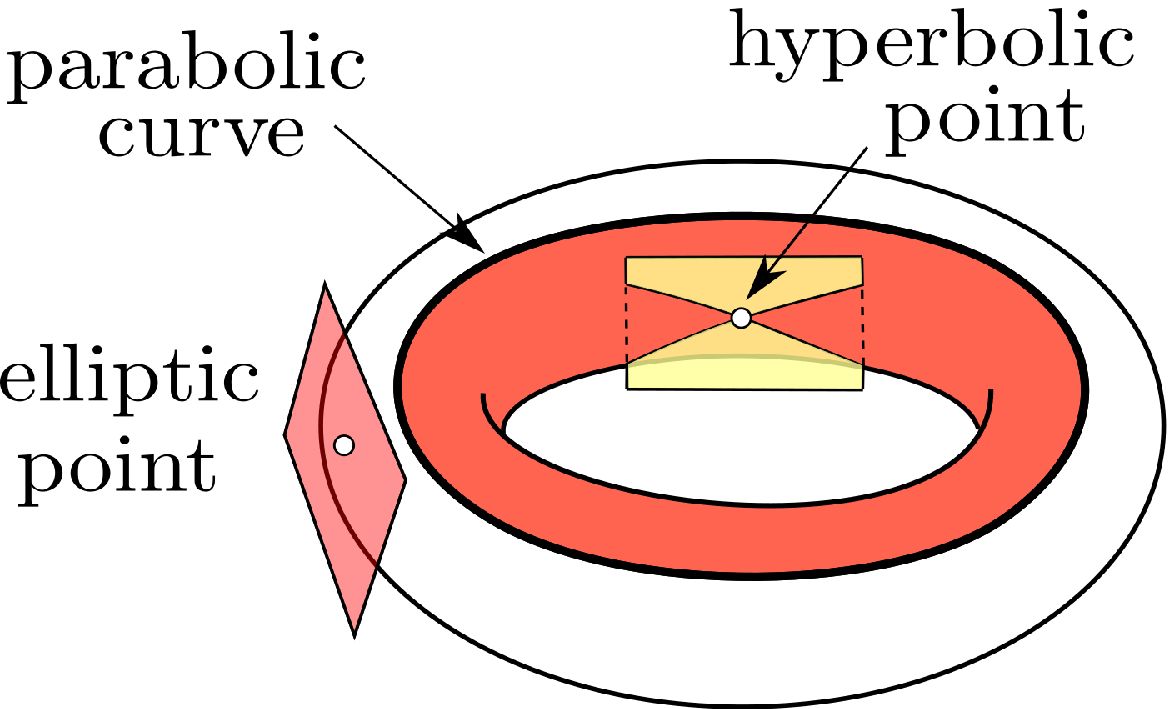}
\caption{\small A generic torus.}
\label{genric-torus}
\end{minipage}
\end{figure}

%%%%%%%%%%%%%%%%%%%%%%%%%%%%%%%%%%%%%%%%%%%%%%%%%%%%%%%%%%%%%%%%%%%%
\paragraph{2.3 Positive and Negative Godrons.} 
%%%%%%%%%%%%%%%%%%%%%%%%%%%%%%%%%%%%%%%%%%%%%%%%%%%%%%%%%%%%%%%%%%%%
A godron is said to be {\em positive} or {\em of index $+1$} 
(resp.~{\em negative} or {\em of index $-1$}) if, at the neighbouring parabolic 
points, the half-asymptotic lines directed to the hyperbolic domain point 
towards (resp.~away from) the godron\,:

\begin{figure}[h] %< préférence de placement h , t , b or p >
\centering
\includegraphics[scale=0.17]{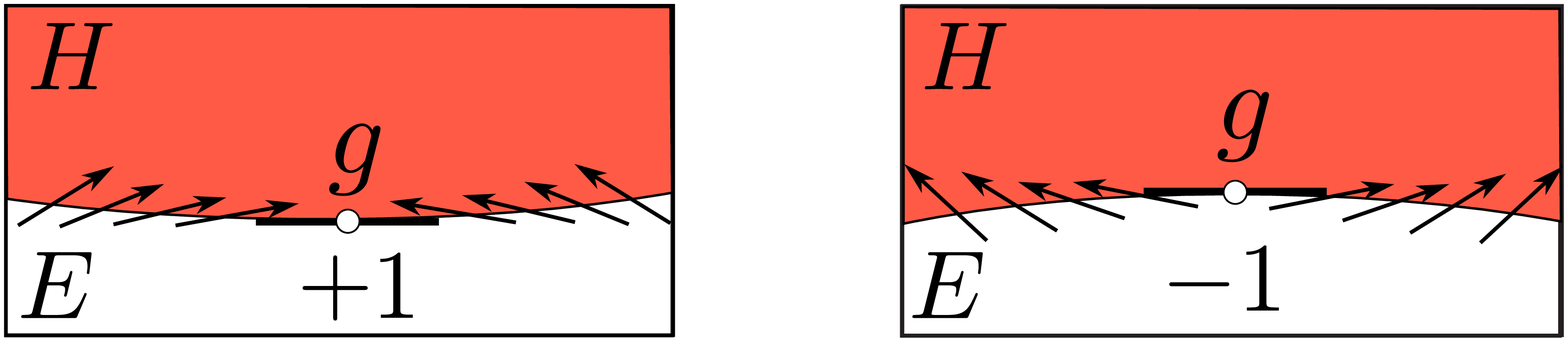}
\caption{\small A positive and a negative godron.}
\label{index-godron}
\end{figure}

Many characterisations and several local and global properties of positive and negative
godrons (and swallowtails) are geometrically described in \cite{Uribegodron}.

%%%%%%%%%%%%%%%%%%%%%%%%%%%%%%%%%%%%%%%%%%%%%%%%%%%%%%%%%%%%%%%%%%%%
\paragraph{2.4 Tangential Map, Godrons and Swallowtails.}\label{fronts}
%%%%%%%%%%%%%%%%%%%%%%%%%%%%%%%%%%%%%%%%%%%%%%%%%%%%%%%%%%%%%%%%%%%%

The {\em tangential map} of a smooth surface $S$, $\tau_S:S\rightarrow (\RP^3)^\vee$,
associates to each point of $S$ its tangent plane at that point.
The image $S^\vee$ of $\tau_S$ is called the {\em dual surface of $S$}.

It is known (cf. \cite{Salmon}) that under the tangential map of $S$ 
{\itshape the parabolic curve of $S$ corresponds to the cuspidal edge of $S^\vee$, 
a godron corresponds to a swallowtail point, and the elliptic (hyperbolic)
  domain of $S$ to the elliptic (resp. hyperbolic) domain of $S^\vee$}.
\medskip

A swallowtail point of a generic front is said to be {\em negative} ({\em positive}) 
if, locally, its self-intersection line is contained in the hyperbolic 
(resp. elliptic) domain. 

\noindent 
\textit{The dual of a surface at a positive (negative) godron is a positive (resp. negative) swallowtail.}
(see Fig.\,\ref{duality})

\begin{figure}[h] %< préférence de placement h , t , b or p >
\centering
\includegraphics[scale=0.32]{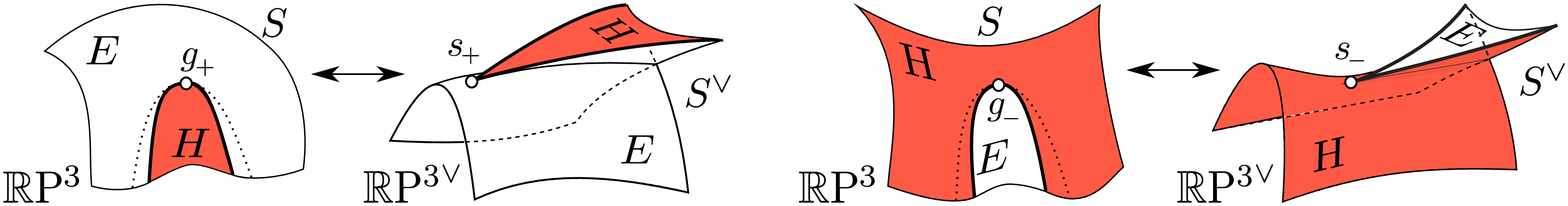}
\caption{\small Duality godron$\,\leftrightarrow\,$swallowtail for positive and negative godrons.}
\label{duality}
\end{figure}

Thus godrons of surfaces of $\RP^3$ ``are the analogs'' of inflections of curves of $\RP^2$.

%%%%%%%%%%%%%%%%%%%%%%%%%%%%%%%%%%%%%%%%%%%%%%%%%%%%%%%%%%%%%%%%%%%%
\paragraph{2.5 Local Transitions and Charateristic Points.}\label{Local-transitions}
%%%%%%%%%%%%%%%%%%%%%%%%%%%%%%%%%%%%%%%%%%%%%%%%%%%%%%%%%%%%%%%%%%%%

If the surface depends on one real parameter (say the time) the configuration formed by the
characteristic points, and the parabolic and flecnodal curves may change. 

Ellipnodes, hyperbonodes and godrons are crucial in the transitions of the parabolic curve of
evolving smooth surfaces, and in the transitions of wave fronts occurring in generic
$1$-parameter families \cite{Uribeevolution}\,:
{\itshape every Morse transition of the parabolic curve takes place at an ellipnode 
  which is replaced by a hyperbonode (or the opposite). At such transition
  two godrons of equal signs born or die} (Fig.\,\ref{metamorphosis}-left).

\begin{figure}[h] %< préférence de placement h , t , b or p >
\centering
\includegraphics[scale=0.37]{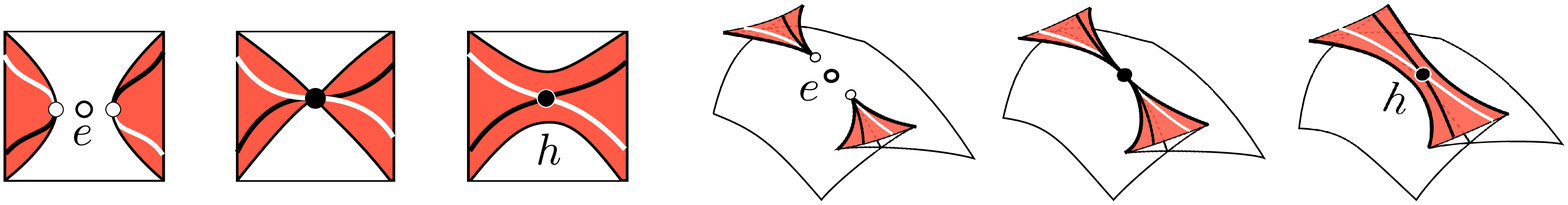}
\caption{\small Examples where an ellipnode is replaced by a hyperbonode or vice-versa.}
\label{metamorphosis}
\end{figure}

Since the planes tangent to $S$ at the points of the flecnodal curve form the flecnodal curve of the 
dual surface $S^\vee$, and the projective dual to a hyperbonode (resp. ellipnode) 
is again a hyperbonode (resp. ellipnode) \cite{Uribetesis, Uribeevolution},  
we get that

  {\itshape Every $A_3$-transition of a wave front, where two swallowtail points 
    born or die, takes place at an ellipnode which is replaced by a
    hyperbonode (or the opposite). The two involved swallowtails have equal signs}
  (Fig.\ref{metamorphosis}-right). 
\medskip

\noindent 
\textbf{Finding the signs}. Corollary\,\ref{3Euler} implies that in both, a Morse transition of the parabolic curve of an evolving
smooth surface and an $A_3$-transition of an evolving wave front, the involved ellipnode and hyerbonode 
have equal signs.

Using Theorem\,\ref{main-theorem} one easily gets the signs of the godrons, ellipnodes and hyperbonodes
that take part in generic local transitions. For example, the hyperbonode of Fig.\,\ref{metamorphosis}-left
is negative because it appears after the Euler characteristic of the hyperbolic domain decreases by $1$,
and, for the same reason, both godrons are positive. 

In the same way, we get the signs of the characteristic points for the local transitions occurring in 
generic $1$-parameter families of smooth surfaces (found in \cite{Uribeevolution})  
in which both, godrons and projective umbilics, are involved (Fig.\,\ref{fig:signs-perestroikas}). 

\begin{figure}[h] %< préférence de placement h , t , b or p >
\centering
\includegraphics[scale=0.3]{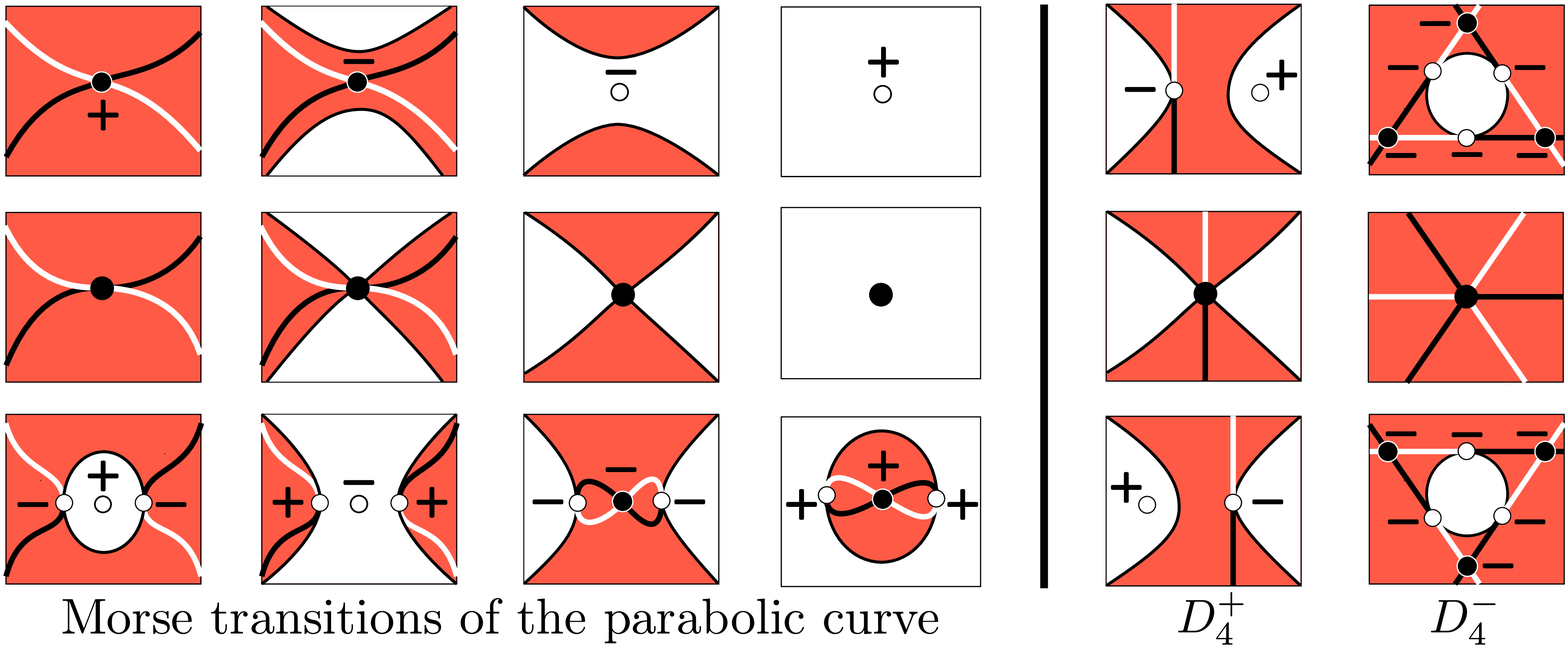}
\caption{\small The signs of godrons and projective umbilics involved in generic local transitions.}
\label{fig:signs-perestroikas}
\end{figure}

%%%%%%%%%%%%%%%%%%%%%%%%%%%%%%%%%%%%%%%%%%%%%%%%%%%%%%%%%%%%%%%%%%%%
\paragraph{2.6 On Normal Forms.}
%%%%%%%%%%%%%%%%%%%%%%%%%%%%%%%%%%%%%%%%%%%%%%%%%%%%%%%%%%%%%%%%%%%%
Although tangential singularities had been studied in the 19th Century by Salmon, Cayley, Zeuthen
et al (cf. \cite{Salmon}), the normal forms of jets of generic surfaces at different kinds of points
was done in the 1980's \cite{Landis, Platonova}, while the normal forms of the
jets (up to order 5) of surfaces at points that appear in generic $1$- and $2$-parameter families of
smooth surfaces were just published in 2017 \cite{Toru}. 

The Landis-Platonova normal form for the $4$-jet at a godron is equivalent to
\begin{equation*}
z=\frac{y^2}{2}-x^2y+\frac{\rho x^4}{2}  \quad (\rho\neq 1)\,,\eqno(\mathcal{P})
\end{equation*}
where $\rho>1$ corresponds to positive godrons and $\rho<1$ to the negative ones.

According to Landis-Platonova's Theorem \cite{Landis, Platonova} and
Ovsienko-Tabachnikov's Theorem \cite{Ovsienko-Tabachnikov}, the $4$-jet of a surface at a
hyperbonode can be sent by projective transformations to the
respective normal forms
\begin{equation*}\label{Landis normal form}
  z=xy+\frac{1}{3!}(ax^3y+bxy^3)+\frac{1}{4!}(x^4\pm y^4)  \quad \ (ab\neq\pm 1)\,, \eqno(\text{L-P})
\end{equation*}
\begin{equation*}
  z=xy+\frac{1}{3!}(x^3y\pm xy^3)+\frac{1}{4!}(Ix^4+Jy^4) \quad \ (IJ\neq\pm 1)\,,\eqno(\text{O-T})
\end{equation*}

To encompass both normal forms we shall consider the ``prenormal'' form
\begin{equation*}
  z=xy+\frac{1}{3!}(ax^3y+bxy^3)+\frac{1}{4!}(Ix^4+Jy^4) \quad \ (IJ\neq ab)\,.)
\end{equation*}
where the genericity conditions on the parameter values ($a,b,I,J$) are imposed in order to avoid the
moment of creation/annihilation of two hyperbonodes. 

In the case of ellipnodes we shall consider the prenormal form
\begin{equation*}
  z=\frac{1}{2}(x^2+y^2)+\frac{1}{3!}(ax^3y+bxy^3)+\frac{1}{4}cx^2y^2+\frac{1}{4!}(Ix^4+Jy^4)\,.)
\end{equation*}
where the corresponding genericity condition is $(a-3b)(b-3a)\neq(I-3c)(J-3c)$. 

%% file: Characteristic_points_on_surfaces2.bbl
{\footnotesize 
  \begin{thebibliography}{9}

\bibitem {Arnold-RPCS} \textbf{Arnold V.I.},
\textit{``Remarks on the Parabolic Curves on Surfaces and on the
  Higher-Dimensional M\"obius-Sturm Theory,''} Funct. Anal. Appl. \textbf{31} (1997) 227-239.

\bibitem {Arnold-LSTSC} \textbf{Arnold V.I.},
\textit{Towards the Legendrian Sturm Theory of Space Curves}, Funct. Anal.
and Appl. \textbf{32}:2 (1998) 75-80.


\bibitem {Arnold-TPTAC} \textbf{Arnold V.I.},
  \textit{``Topological Problems of the Theory of Asymptotic Curves,''}
  Proc. Steklov Inst. Math. \textbf{225} (1999) 5-15.
  
\bibitem {Chekanov-Pushkar} \textbf{Chekanov Yu. V.}, \textbf{Pushkar P. E.}, 
  \textit{``Combinatorics of Fronts of Legendrian Links and the Arnol’d 4-Conjectures,''}
  Russ. Math. Surv. \textbf{60} (2005) 95-149.

\bibitem {Kazarian-Flat} \textbf{Kazarian M.}, \textit{Nonlinear Version of Arnold’s
Theorem on Flattening Points}, C.R. Acad. Sci. Paris, t.\,\textbf{323}, Série I, no.1, (1996) 63-68.

\bibitem{Kazarian} {\bf Kazarian M.}, {\em Chern-Euler number of circle bundles via singularity theory}.
Math. Scand., {\bf 82} (1998), 207-236.

\bibitem {Landis} {\bf Landis E.E.}, {\em Tangential singularities},
  Funct. Anal. Appl. {\bf 15}:2 (1981), 103-114.
  
\bibitem{NS} \textbf{Nomizu K.}, \textbf{Sasaki T.},
\textit{Affine Differential Geometry}. Tracts in Math.\,\textbf{111}, Cambridge U.\,P.\,(1994).

\bibitem {Ovsienko-Tabachnikov} \textbf{Ovsienko V.}, \textbf{Tabachnikov S.}, 
  \textit{Hyperbolic Carathéodory Conjecture}. Proc. of Steklov Inst. of Mathematics,
  \textbf{258} (2007) 178-193.

\bibitem {Dima} \textbf{Panov D.A.}, {\em Special Points of Surfaces in the
Three-Dimensional Projective Space},
Funct. Anal. Appl. {\bf 34}:4 (2000) 276-287.

\bibitem {Platonova} {\bf Platonova O.A.}, {\em Singularities of the mutual
position of a surface and a line}, Russ. Math. Surv. {\bf 36}:1 (1981)
248-249. Zbl.458.14014.

\bibitem {Salmon} {\bf Salmon G.}, \textit{A Treatise in Analytic
  Geometry of Three Dimensions}, Chelsea Publ. (1927).

\bibitem {Toru} \textbf{Sano H.}, \textbf{Kabata Y.}, \textbf{Deolindo Silva J.}, \textbf{Ohmoto T.},
  \textit{Classification of Jets of Surfaces in Projective $3$-Space Via Central Projection},
  Bull. Math. Braz. Soc. (2017), DOI 10.1007/s00574-017-0036-x.

\bibitem {Uribeconformal} \textbf{Uribe-Vargas R.},  \textit{Four-Vertex Theorems in Higher
  Dimensional Spaces for a Larger Class of Curves than the Convex Ones},
  C.R. Acad. Sci. Paris, t.\,\textbf{330}, Série I, (2000) 1085–1090.

\bibitem {Uribetesis} \textbf{Uribe-Vargas R.}, {\em Singularit\'es symplectiques et de 
contact en g\'eom\'etrie diff\'erentielle des courbes et des surfaces}, 
  PhD. Thesis, Universit\'e Paris 7, 2001, (In English).

\bibitem {Uribeevolution} \textbf{Uribe-Vargas R.}, {\em Surface Evolution, Implicit 
Differential Equations and Pairs of Legendrian Fibrations}, Preprint (2002). 
An improved version has been submitted for publication (2019). 

\bibitem {Uribegodron} \textbf{Uribe-Vargas R.}, {\em A Projective Invariant for
Swallowtails and Godrons, and Global Theorems on the Flecnodal Curve}, 
  Moscow Math. Jour. \textbf{6}:4 (2006) 731-772.

\bibitem {Uribeinvariant} \textbf{Uribe-Vargas R.},
{\em On Projective Umbilics: a Geometric Invariant and an Index}.
Journal of Singularities \textbf{17}, Worldwide Center of Mathematics, LLC (2018) 81-90. DOI 10.5427/jsing.2018.17e

\end{thebibliography}
}
